\documentclass{amsart}
\usepackage{amsmath, amsthm, amsfonts, amssymb, txfonts, graphicx, verbatim, longtable,hyperref}
\theoremstyle{plain}
\newtheorem{thrm}{Theorem}
\newtheorem{lmm}[thrm]{Lemma}
\newtheorem{prpstn}[thrm]{Proposition}

\newtheorem*{prpstn*}{Proposition}
\newtheorem*{thm*}{Theorem}

\newtheorem*{lmm*}{Lemma}
\newtheorem*{lmmA1*}{Lemma A1}
\newtheorem*{lmmA2*}{Lemma A2}
\newtheorem*{sblmm*}{Sublemma}
\newtheorem{cndtn}{Condition}
\numberwithin{equation}{section}
\DeclareMathOperator{\Li}{Li}
\begin{document}
\title{On the Brun-Titchmarsh theorem}
\author{James Maynard}
\address{Mathematical Institute, 24–-29 St Giles', Oxford, OX1 3LB}
\email{maynard@math.ox.ac.uk}
\thanks{Supported by EPSRC Doctoral Training Grant EP/P505216/1 }
\date{}
\begin{abstract}
The Brun-Titchmarsh theorem shows that the number of primes which are less than $x$ and congruent to $a\pmod{q}$ is less than $(C+o(1))x/(\phi(q)\log{x})$ for some value $C$ depending on $\log{x}/\log{q}$. Different authors have provided different estimates for $C$ in different ranges for $\log{x}/\log{q}$, all of which give $C>2$ when $\log{x}/\log{q}$ is bounded. We show that one can take $C=2$ provided that $\log{x}/\log{q}\ge8$ and $q$ is sufficiently large. Moreover, we also produce a lower bound of size $x/(q^{1/2}\phi(q))$ when $\log{x}/\log{q}\ge 8$ and is bounded. Both of these bounds are essentially best-possible without any improvement on the Siegel zero problem.
\end{abstract}
\maketitle
\section{Introduction}
We let $\pi(x;q,a)$ denote the number of primes less than or equal to $x$ which are congruent to $a\pmod{q}$, for some real $x>0$ and positive coprime integers $a,q$. It is a classical theorem of Walfisz \cite{Walfisz} based on the work of Siegel that, for any fixed $N>0$, uniformly for $q\le (\log{x})^N$ and $(a,q)=1$, as $x\rightarrow \infty$ we have
\begin{equation}\label{SiegelWalfisz}
\pi(x;q,a)\sim \frac{x}{\phi(q)\log{x}}.
\end{equation}
It is generally believed that this asymptotic holds in a much wider range of $q$. If we assume the generalised Riemann Hypothesis (GRH), then the asymptotic \eqref{SiegelWalfisz} holds uniformly in the much larger range $q\le x^{1/2-\delta}$ for any fixed $\delta>0$. Montgomery \cite{MontgomeryBook} has conjectured that the asymptotic holds uniformly in the even larger range $q\le x^{1-\delta}$.

Any improvement in the range of $q$ for which the asymptotic holds would exclude the possibility of the existence of zeros of Dirichlet $L$-functions in certain regions, but unfortunately such a result seems beyond our current techniques. Without this type of improvement, however, we cannot hope to prove results stronger than
\begin{equation}o\left(\frac{x}{\phi(q)\log{x}}\right)\le \pi(x;q,a)\le \frac{2x}{\phi(q)\log{x}}\end{equation}
when $\log{x}/\log{q}$ is bounded.

Linnik \cite{Linnik1}, \cite{Linnik2} gave a non-trivial lower bound for $\pi(x;q,a)$ for a wider range of $q$. He showed that there is a constant $L>0$ such that, whenever $x>q^L$ and $q$ is sufficiently large there is at least one prime in the arithmetic progression $\{n\le x:n\equiv a\pmod{q}\}$ for any $a$ with $(a,q)=1$. Pan \cite{Pan} showed that one can take $L\le 10,000$. This has subsequently been improved by many authors including (in chronological order) Chen \cite{Chen1}, Jutila \cite{Jutila1}, Chen \cite{Chen2}, Jutila \cite{Jutila},Chen \cite{Chen3}, Graham \cite{Graham1}, Wang \cite{Wang}, Chen and Liu \cite{ChenLiu}, and Heath-Brown \cite{HB}. The best known result is due to Xylouris \cite{Xylouris}, which shows that we can take $L=5.2$.

Titchmarsh \cite{Titchmarsh} used Brun's sieve to show that for $q< x$ we have the upper bound
\begin{equation}
\pi(x;q,a)\ll \frac{x}{\phi(q)\log{(x/q)}}.
\end{equation}
The implied constant can be made explicit, and has been estimated by various authors. The strongest result of this type which holds for all ranges of $q$ is due to Montgomery and Vaughan \cite{MontgomeryVaughan}, who used the large sieve to obtain the following result.
\begin{prpstn*}[Brun-Titchmarsh Theorem]
For $x>q$ we have
\[\pi(x;q,a)\le \left(\frac{2}{1-\log{q}/\log{x}}\right)\frac{x}{\phi(q)\log{x}}.\]
\end{prpstn*}
The constant $2/(1-\log{q}/\log{x})$ of the Brun-Titchmarsh theorem should be compared with the constant $1+o(1)$ which Montgomery conjectures.

Since it appears unlikely that we can prove an upper bound with a constant less than 2 with the current techniques, any improvements are likely to reduce the factor $1/(1-\log{q}/\log{x})$. Several authors including Motohashi \cite{MotohashiI}, Goldfeld \cite{Goldfeld}, Iwaniec \cite{Iwaniec1} and Iwaniec and Friedlander \cite{Iwaniec2}  have made improvements of this type for different ranges of $q$. If we put
\begin{equation}
\theta=\frac{\log{q}}{\log{x}},
\end{equation}
then we have
\begin{equation}
\pi(x;q,a)\le \frac{(C+o(1))x}{\phi(q)\log{x}},
\end{equation}
where
\begin{equation}
C=
\begin{cases}
(2-((1-\theta)/4)^6)/(1-\theta),\qquad &2/3\le \theta,\nonumber\\
8/(6-7\theta),&9/20\le\theta\le 2/3,\\
16/(8-3\theta),&\theta\le 9/20.\nonumber
\end{cases}
\end{equation}
This improves the Brun-Titchmarsh bound of $C= 2/(1-\theta)$ slightly throughout the entire range of $q$. We note that in all cases we still have $C>2$ for $\theta>0$.

It has been known as a folklore amongst specialists that for $\theta$ less than some fixed constant we should be able to take $C=2$. In this paper we establish this, and give a quantitative bound for the range when this happens. We show that provided $q$ is sufficiently large we can take
\[C=2\qquad \text{if }\theta\le 1/8.\]
\section{Notation}
We will let $p$ represent a generic prime. We will consider the arithmetic progression where all terms are $\le x$ and are congruent to $a\pmod{q}$. We will assume that $q$ is larger than some fixed constant throughout, and so may not explicitly say that we are assuming $q$ to be sufficiently large for a given statement to hold. $\chi$ will refer to a Dirichlet character (mod $q$) and $\chi_0$ the principal character.

For the purposes of this paper we shall define an `$\eta$-Siegel zero' to be a real zero $\rho$ of a Dirichlet $L$-function $L(s,\chi)$ which lies in the region
\[1-\frac{\eta}{\log{q}}\le \Re(\rho)\le 1.\]
\section{Main Result}
We improve on the Brun-Titchmarsh constant for some range of $q$. Instead of using sieve methods to count primes in arithmetic progressions we will use the analytic techniques developed in the estimation of Linnik's constant.

In Linnik's theorem one counts primes with a smooth weight, and estimating this requires estimating corresponding weighted sums over the zeros of Dirichlet $L$-functions. In the case of Linnik's theorem only zeros of the form $\rho=1+O(1/\log{q})$ make a significant contribution. In this paper we wish to count primes weighted by the characteristic function of the interval $[0,x]$, however, and this means we must consider all zeros $\rho=\beta+i\gamma$ with $\gamma\ll 1$ in the corresponding weighted sums over zeros. Thus the zero density estimates of Heath-Brown \cite{HB} are insufficient, and we need to extend them to this larger range.

\begin{thrm}\label{thrm:MainTheorem} There exists an effectively computable constant $q_1$, such that for $q\ge q_1$ and $x\ge q^8$ we have
\[\pi(x;q,a)< \frac{2\Li(x)}{\phi(q)}.\]
\end{thrm}
We note that without excluding the possible existence of $\eta$-Siegel zeros for some $\eta>0$ this is the strongest possible bound which we can hope to prove for $\log{x}/\log{q}$ bounded.

We also obtain lower bounds which are essentially the strongest possible  for $\log{x}/\log{q}$ bounded without excluding the existence of an $\eta$-Siegel zero.
\begin{thrm}\label{thrm:LowerBoundEffective}
There exists an effectively computable constant $q_2$ such that for $q\ge q_2$ and $x\ge q^8$ we have
\[\frac{\log{q}}{q^{1/2}}\left(\frac{x}{\phi(q)\log{x}}\right)\ll \pi(x;q,a).\]
\end{thrm}
\begin{thrm}\label{thrm:LowerBoundIneffective}
Let $\epsilon>0$. There exists an (ineffective) constant $q_3(\epsilon)$ such that for $q\ge q_3(\epsilon)$ and $x \ge q^{8}$ we have
\[\frac{q^{-\epsilon}x}{\phi(q)\log{x}}\ll\pi(x;q,a).\]
\end{thrm}
\begin{thrm}\label{thrm:NoSiegelZero}
Assume that there exists a constant $\eta>0$ such that there are no $\eta$-Siegel zeros. Then there exists an effectively computable constant $q_4$ such that for $q\ge q_4$ and $x\ge q^8$ we have
\[\frac{x}{\phi(q)\log{x}}\ll \pi(x;q,a)< \frac{2x}{\phi(q)\log{x}}.\]
\end{thrm}
Thus the number of primes in an arithmetic progression is close to expected order predicted by GRH, provided $\log{x}/\log{q}\ge 8$ and $q$ is sufficiently large. If there are no zeros exceptionally close to 1 then the number of primes has the same order as the asymptotic predicted by GRH.

In order to establish Theorems \ref{thrm:MainTheorem}, \ref{thrm:LowerBoundEffective}, \ref{thrm:LowerBoundIneffective} and \ref{thrm:NoSiegelZero} we prove the following proposition.
\begin{prpstn}\label{Prpstn:OverallResult}
There are fixed constants $\epsilon>0$ and $\eta>0$ such that:

If there is an $\eta$-Siegel zero $\rho_1=1-\lambda_1/\log{q}$ then there exists an effectively computable constant $q_5$, such that for $q\ge q_5$ and $x\ge q^{7}$ we have
\[\left|\psi(x;q,a)-\frac{x}{\phi(q)}\right|<\frac{(1-\lambda_1)x}{\phi(q)}.\]
If there are no $\eta$-Siegel zeros then there exists an effectively computable constant $q_6$ such that for $q\ge q_6$ and for $x\ge q^{7.999}$ we have
\[\left|\psi(x;q,a)-\frac{x}{\phi(q)}\right|<\frac{(1-\epsilon)x}{\phi(q)}.\]
\end{prpstn}
We now establish Theorems \ref{thrm:MainTheorem}, \ref{thrm:LowerBoundEffective}, \ref{thrm:LowerBoundIneffective} and \ref{thrm:NoSiegelZero} assuming Proposition \ref{Prpstn:OverallResult}.

By partial summation we have for any constant $7\le A< 8$
\begin{align}
\pi(x;q,a)&=\frac{\theta(x;q,a)}{\log{x}}+\int_2^x \frac{\theta(t;q,a)}{t\log^2{t}}d t\nonumber\\
&=\frac{\theta(x;q,a)}{\log{x}}+\int_{q^{A}}^x \frac{\theta(t;q,a)}{t\log^2{t}}d t+\int_{q^2}^{q^{A}} \frac{\theta(t;q,a)}{t\log^2{t}}dt+\int_{2}^{q^2} \frac{\theta(t;q,a)}{t\log^2{t}}d t.
\end{align}
By the Brun-Titchmarsh Theorem for $q^2\le t\le q^{A}$ we have
\begin{equation}
\theta(t;q,a)\le (\log{t})\pi(t;q,a)\ll \frac{t}{\phi(q)},
\end{equation}
and trivially for $t\le q^2$ we have
\begin{equation}
\theta(t;q,a)\le t\log{t}.
\end{equation}
We also note that
\[\theta(x;q,a)=\psi(x;q,a)+O(x^{1/2}).\]
Thus we have uniformly for $x\ge q^8$ and $7\le A\le 8$ that
\begin{align}
\pi(x;q,a)&=\frac{\psi(x;q,a)}{\log{x}}+\int_{q^{A}}^x \frac{\psi(t;q,a)}{t\log^2{t}}dt+O\left(x^{1/2}+\frac{q^{A}}{\phi(q)}\right).
\end{align}
This gives
\begin{align}
\left|\pi(x;q,a)-\frac{\Li(x)}{\phi(q)}\right|&\le \frac{1}{\log{x}}\left|\psi(x;q,a)-\frac{x}{\phi(q)}\right|+\int_{q^A}^x\frac{\left|\psi(t;q,a)-t/\phi(q)\right|}{t\log^2{t}}dt\nonumber\\
&\qquad+O\left(x^{1/2}+\frac{q^A}{\phi(q)}\right).\label{eq:MainPiExpression}
\end{align}
If there is an $\eta$-Siegel zero (where $\eta$ is the constant from Proposition \ref{Prpstn:OverallResult}) then we choose $A=7$ and by Proposition \ref{Prpstn:OverallResult} uniformly for $q\ge q_6$ and $x\ge q^8$ we have
\begin{align}
\left|\pi(x;q,a)-\frac{\Li(x)}{\phi(q)}\right|&\le \frac{(1-\lambda_1)x}{\phi(q)\log{x}}+\int_{q^{7}}^x\frac{1-\lambda_1}{\phi(q)\log^2{t}}dt+O\left(x^{1/2}+\frac{q^{7}}{\phi(q)\log{x}}\right)\nonumber\\ 
&\le\frac{(1-\lambda_1)\Li(x)}{\phi(q)}+O\left(\frac{x}{q\phi(q)\log{x}}\right).
\end{align}
By Pintz \cite{Pintz}[Theorem 3] we have that $\lambda_1\gg \log{q}/q^{1/2}$ (with the implied constant effectively computable).

Thus for $q$ sufficiently large and $x\ge q^8$ we have
\begin{equation}
\frac{x\log{q}}{q^{1/2}\phi(q)\log{x}}\ll \frac{\lambda_1\Li(x)}{2\phi(q)\log{x}}\le \pi(x;q,a)\le \frac{(2-\lambda_1)\Li(x)}{\phi(q)}\le \frac{2\Li(x)}{\phi(q)},\label{eq:SiegelUpperBound}
\end{equation}
with all constants effectively computable.

By Siegel's theorem \cite{Siegel}, given any $\epsilon>0$ there is a constant $C(\epsilon)$ such that if $q\ge C(\epsilon)$ we have $\lambda_1\ge 2q^{-\epsilon}$. Here the constant $C(\epsilon)$ is not effectively computable. In this case, we have
\begin{align}
\frac{x q^{-\epsilon}}{\phi(q)\log{x}}\le \frac{\lambda_1\Li(x)}{2\phi(q)\log{x}}<\pi(x;q,a).\label{eq:IneffectiveBound}
\end{align}
If there is no $\eta$-Siegel zero then we instead choose $A=7.999$. By Proposition \ref{Prpstn:OverallResult} and \eqref{eq:MainPiExpression} there exists an $\epsilon>0$ and $q_5$ such that uniformly for $x\ge q^8$ and for $q\ge q_5$ we have 
\begin{align}
\left|\pi(x;q,a)-\frac{x}{\phi(q)}\right|&\le \frac{(1-\epsilon)x}{\phi(q)\log{x}}+\int_{q^{7}}^x\frac{1-\epsilon}{\phi(q)\log^2{t}}dt+O\left(x^{1/2}+\frac{q^{7.999}}{\phi(q)\log{x}}\right)\nonumber\\
&=\frac{(1-\epsilon)\Li(x)}{\phi(q)}+O\left(\frac{x^{1-1/10,000}}{\phi(q)\log{x}}\right).
\end{align}
Thus for $q$ sufficiently large and $q^8\le x$ we have
\begin{equation}
\frac{x}{\phi(q)\log{x}}\ll \pi(x;q,a)< \frac{2x}{\phi(q)\log{x}}.\label{eq:NonSiegelUpperBound}
\end{equation}
Theorems \ref{thrm:MainTheorem}, \ref{thrm:LowerBoundEffective}, \ref{thrm:LowerBoundIneffective} and \ref{thrm:NoSiegelZero} now follow immediately from \eqref{eq:SiegelUpperBound}, \eqref{eq:IneffectiveBound} and \eqref{eq:NonSiegelUpperBound}.
\section{Case 1: Siegel Zeroes}\label{sec:Siegel}
We first consider the case when there are zeros very close to 1. For this section we assume that $\eta$-Siegel zeros exist for some small constant $\eta>0$.

In order to establish Proposition \ref{Prpstn:OverallResult} we will make use of the analytic techniques developed in the estimation of Linnik's constant. In particular, there are three main results which we use:
\begin{prpstn}[Zero-free region]\label{Prpstn:ZeroFree}
There is a constant $c_1>0$ such that for $q$ sufficiently large
\[\prod_{\chi\pmod{q}}L(\sigma+it,\chi)\]
has at most one zero in the region
\[1-\frac{c_1}{\log{q(2+|t|)}}\le \sigma.\]
Such a zero, if it exists, is real and simple, and the corresponding character must be a non-principal real character.
\end{prpstn}
\begin{prpstn}[Deuring-Heilbronn phenomenon]\label{Prpstn:DeuringHeilbronn}
There is a constant $c_2>0$ such that, if the exceptional zero $\rho_1=1-\lambda_1/(\log{q})$ from Proposition \ref{Prpstn:ZeroFree} exists, then for $q$ sufficiently large, the function
\[\prod_{\chi\pmod{q}}L(\sigma+it,\chi)\]
has no other zeros in the region
\[1-\frac{c_2\log(\lambda_1^{-1})}{\log{q(2+|t|)}}\le \sigma\le 1.\]
\end{prpstn}
\begin{prpstn}[Log-free zero-density estimate]\label{Prpstn:ZeroDensity}
For $T\ge 1$ there are constants $c_3>0$ and $C_3>0$ such that
\[\sum_{\chi \pmod{q}}N(\sigma,T,\chi)\le C_3(q T)^{c_3(1-\sigma)}.\]
Here
\[N(\sigma,T,\chi)=\#\{\rho:L(\rho,\chi)=0,\quad\Re(\rho)\ge \sigma,\quad|\Im(\rho)|\le T\}.\] 
\end{prpstn}
We recall that for the purposes of this article we are defining a $\eta$-Siegel zero to be a real zero $\rho$ of some Dirichlet $L$-function in the region
\begin{equation}
1-\frac{\eta}{\log{q}}\le \rho\le 1
\end{equation}
for a fixed small positive constant $\eta$.

We will choose $\eta\le c_1/2$, so by Proposition \ref{Prpstn:ZeroFree} a $\eta$-Siegel zero, if it exists, must be simple, and the corresponding character must be a real character. Moreover, there can be at most one such zero. We label this exceptional zero $\rho_1=1-\lambda_1/(\log{q})$ with corresponding character $\chi_1$. Thus we have that $\lambda_1\le \eta$. We will also make use of the fact that $\lambda_1\gg_\epsilon q^{-1/2-\epsilon}$ (with the implied constant effectively computable), which follows from Dirichlet's class number formula.

We note that by \cite{GrahamThesis} and \cite{HB}[Equation 1.4] we can take
\begin{equation}
c_2=2/3-1/1000,\qquad c_3=12/5+1/1000,
\end{equation}
provided $\eta\le c_4$, some suitably small absolute constant.

We wish to prove
\begin{equation}\label{aim}
\left|\psi(x;q,a)-\frac{x}{\phi(q)}\right|\le \frac{(1-\lambda_1)x}{\phi(q)}.
\end{equation}
We have that
\begin{align}
\psi(x;q,a)&=\sum_{\substack{n\le x\\n\equiv a \pmod{q}}}\Lambda(n)\nonumber\\
&=\frac{1}{\phi(q)}\sum_{\chi \pmod{q}}\overline{\chi(a)}\left(\sum_{n\leq x} \Lambda(n)\chi(n)\right).\label{eq:PsiExpression}
\end{align}
We use the explicit formula: 
\begin{equation}\sum_{n\leq x}\Lambda(n)\chi(n)=\varepsilon_1(\chi)x-\varepsilon_2(\chi)\frac{x^{\rho_1}}{\rho_1}-\sum_\rho\frac{x^\rho}{\rho}+O\left(\frac{x(\log{x})^2}{T}\right),\label{eq:ExplicitFormula}\end{equation}
where:
\begin{align}
\varepsilon_1(\chi)&=
\begin{cases}
1,\qquad &\chi=\chi_0,\\
0, &\text{otherwise},
\end{cases}\\
\varepsilon_2(\chi)&=
\begin{cases}
1,\qquad &\chi \text{ is a character corresponding to the possible}\\
&\qquad\qquad\qquad\text{exceptional zero }\rho_1\text{ of }\prod_\chi L(s,\chi),\\
0, &\text{otherwise},
\end{cases}
\end{align}
and the sum $\sum_\rho$ is over all non-exceptional non-trivial zeros $\rho = \beta+i\gamma$ of $L(s,\chi)$ in the region $\{0<\beta<1,|\gamma|<T\}$.

We choose $T=q(\log{x})^3/\lambda_1$ so that the last term is $o(\lambda_1x/\phi(q))$.

Recalling that $\rho_1=1-\lambda_1/\log{q}$ we have
\begin{equation}
\frac{x^{\rho_1}}{\rho_1}=x\exp\left(-\lambda_1\frac{\log{x}}{\log{q}}\right)+o(\lambda_1 x).\label{eq:ExceptionalZeroSize}
\end{equation}
Substituting \eqref{eq:ExplicitFormula} and \eqref{eq:ExceptionalZeroSize} into \eqref{eq:PsiExpression} we have
\begin{equation}
	\left|\psi(x;q,a)-\frac{x}{\phi(q)}\right|\le \frac{x}{\phi(q)}\exp\left(-\lambda_1\frac{\log{x}}{\log{q}}\right)+\frac{1}{\phi(q)}\sum_{\chi\pmod{q}}\sum_{\rho}\left|\frac{x^\rho}{\rho}\right|+o\left(\frac{\lambda_1 x}{\phi(q)}\right).\label{eq:MainPsiExpression}
\end{equation}
We now bound the inner sum
\begin{equation}
\sum_{\chi \pmod{q}}\sum_{\rho}\left|\frac{x^\rho}{\rho}\right|.\label{eq:SiegelInnerSum}
\end{equation}
We first consider the case when $\log{x}>q^{1/3000}$. 

Since $\lambda_1\gg q^{-1/2-1/100}$ we have $T\ll q^{3/2+1/100}(\log{x})^3\ll (\log{x})^{4600}$. By Proposition \ref{Prpstn:ZeroFree} (and recalling $|\rho|\gg \lambda_1/\log{q}$ for all $\rho$) each zero in the sum \eqref{eq:SiegelInnerSum} contributes at most
\begin{equation}
\left|\frac{x^\rho}{\rho}\right|\le x\exp\left(-c\frac{\log{x}}{\log\log{x}}\right)
\end{equation}
for some constant $c>0$. By Proposition \ref{Prpstn:ZeroDensity} the total number of zeros in the sum is
\begin{equation}
\ll (qT)^{12/5+1/1000}\ll (\log{x})^{20000}.
\end{equation}
Thus we have that
\begin{equation}
\sum_{\chi \pmod{q}}\sum_{\rho}\left|\frac{x^\rho}{\rho}\right|\ll x(\log{x})^{20000}\exp\left({-c\frac{\log{x}}{\log\log{x}}}\right)= o(\lambda_1 x).
\end{equation}
Thus we see that for $x$ sufficiently large and $\log{x}>q^{1/3000}$, the right hand side of \eqref{eq:MainPsiExpression} is
\begin{equation}
\frac{x}{\phi(q)}\left(\exp\left(-\lambda_1\frac{\log{x}}{\log{q}}\right)+o(\lambda_1)\right)\le \frac{(1-\lambda_1)x}{\phi(q)},
\end{equation}
as required.

We now consider the case when $\log{x}\le q^{1/3000}$. In this case, since $\lambda_1\gg q^{-1/2-1/1000}$ we have $T\ll q^{3/2+2/1000}$.

We first consider the contribution to the sum \eqref{eq:SiegelInnerSum} from zeros in the rectangle
\begin{equation}
1-\frac{m+1}{\log{q}}\le \Re(\rho) \le 1-\frac{m}{\log{q}},\qquad n\le |\Im(\rho)|\le 2n,
\end{equation}
where $1\le n\le T$ and $m\le 0.4\log{q}$. By Proposition \ref{Prpstn:DeuringHeilbronn} with $c_2=2/3-1/1000$ there are no zeros in the rectangle unless
\begin{align}
m&\ge \left(\frac{2}{3}-\frac{1}{1000}\right)\left(\frac{\log{q}}{\log{q(2+T)}}\right)\log{\lambda_1^{-1}}\nonumber\\
&\ge 0.266\log{\lambda_1^{-1}}.\label{eq:MUpperBound}
\end{align}
Recalling that $m\le 0.4\log{q}$, by Proposition \ref{Prpstn:ZeroDensity} with $c_3=12/5+1/1000$ there are
\begin{equation}
\ll n^{0.97}\exp\left(2.41m\right)
\end{equation}
zeros in the rectangle.

If \eqref{eq:MUpperBound} holds then we see that each zero contributes
\begin{align}
\left|\frac{x^\rho}{\rho}\right|&\le \frac{x}{n}\exp\left(-m\frac{\log{x}}{\log{q}}\right)\nonumber\\
&\le  \frac{x}{n}\exp\left(-m\left(\frac{\log{x}}{\log{q}}-\frac{1}{0.266}\right)\right)\exp\left(-\frac{m}{0.266}\right)\nonumber\\
&\le \frac{\lambda_1x}{n}\exp\left(-m\left(\frac{\log{x}}{\log{q}}-3.76\right)\right).
\end{align}
Thus zeros in the rectangle give a total contribution of
\begin{equation}
\ll \frac{\lambda_1 x}{n^{0.03}}\exp\left(-m\left(\frac{\log{x}}{\log{q}}-6.17\right)\right).
\end{equation}
From summing this bound, we see that provided $q^{6.18}\le x$, the contribution to the sum \eqref{eq:SiegelInnerSum} from all non-exceptional zeros in the region
\begin{equation}
0.6\le \Re(\rho)\le 1,\qquad 1\le |\Im(\rho)|\le T
\end{equation}
is at most
\begin{equation}
C\lambda_1 x \exp(-c\log{\lambda_1}^{-1})\le C\lambda_1 x \exp(c\log{\eta})
\end{equation}
for some constants $C,c>0$. Since $\lambda_1\le \eta$ we see that for $\eta$ sufficiently small (depending only on $C,c$) this is at most $\lambda_1 x$.

Similarly we consider the contribution to the sum \eqref{eq:SiegelInnerSum} from zeros in the region
\begin{equation}
1-\frac{m+1}{\log{q}}\le \Re(\rho)\le 1-\frac{m}{\log{q}},\qquad |\Im(\rho)|\le 1,
\end{equation}
with $m\le 0.4\log{q}$. As above, each zero contributes
\begin{equation}\ll \lambda_1 x\exp\left(-m\left(\frac{\log{x}}{\log{q}}-3.76\right)\right).\end{equation}
The number of zeros in the rectangle is
\begin{equation}
\ll \exp(2.41m).
\end{equation}
Thus again the contribution of all zeros from the rectangles is at most
\begin{equation}
C\lambda_1 x \exp(-c\log{\lambda_1}^{-1})\le C\lambda_1 x \exp(c\log{\eta})
\end{equation}
for some positive constants $C,c$. Thus for $\eta$ sufficiently small this contribution is at most $\lambda_1x$.

Finally we consider zeros in the rectangles
\begin{equation}
0\le \Re(\rho)\le 0.6,\qquad |\Im(\rho)|\le \sqrt{T}\label{eq:R1}
\end{equation}
and
\begin{equation}
0\le \Re(\rho)\le 0.6,\qquad \sqrt{T}\le |\Im(\rho)|\le T\label{eq:R2}.
\end{equation}
By symmetry of zeros around the line $\Re(s)=1/2$ we have that $\Re(\rho)\gg \lambda_1/\log{q}$ for all such $\rho$. Thus, since $\lambda_1\gg q^{-1/2-1/100}$ and $x>q$ each zero satisfying \eqref{eq:R1}  contributes
\begin{equation}
\left|\frac{x^\rho}{\rho}\right| \le x^{0.6},
\end{equation}
and every zero satisfying \eqref{eq:R2} contributes
\begin{equation}
\left|\frac{x^\rho}{\rho}\right| \le \frac{x^{0.6}}{\sqrt{T}}.
\end{equation}
For $q$ sufficiently large there are
\begin{equation}
\ll (q\sqrt{T})^{1+1/1000}\le q^{1.76}
\end{equation}
zeros satisfying \eqref{eq:R1}, and
\begin{equation}
\ll (qT)^{1+1/1000}\le q^{1.76}\sqrt{T}
\end{equation}
zeros satisfying \eqref{eq:R2}. Thus the combined contribution is
\begin{equation}
\ll x^{0.6}q^{1.76}\ll \lambda_1 x \left(\frac{q^{2.27}}{x^{0.4}}\right).
\end{equation}
we see this is at most $\lambda_1 x$ for $q^{6}\le x$ and $q$ sufficiently large.

Since we have now covered all possible zeros in our sum, we see that for $\eta$ sufficiently small and $q^{6.18}\le x$ we have
\begin{equation}
\sum_{\chi\pmod{q}} \sum_{\rho}\left|\frac{x^\rho}{\rho}\right|\le 3\lambda_1 x.
\end{equation}
Substituting this into \eqref{eq:MainPsiExpression} we see that
\begin{equation}
\left|\psi(x;q,a)-\frac{x}{\phi(q)}\right|\le \frac{x}{\phi(q)}\left(\exp\left(-\lambda_1 \frac{x}{\phi(q)}\right)+4\lambda_1 \right).\end{equation}
We note that if $q^7\le x$ and $\eta<1/10$ then we have
\begin{equation}
\exp\left(-\lambda_1 \frac{\log{x}}{\log{q}}\right)+4\lambda_1 < 1-\lambda_1,
\end{equation}
since $1-e^{-7t}-5t$ is zero and increasing at 0, has a unique turning point and is positive at 1/10.

Thus we have shown that for $\eta$ sufficiently small, $q^{7}\le x$ and $\log{x}\le q^{3000}$ we have
\begin{equation}
\left|\psi(x;q,a)-\frac{x}{\phi(q)}\right|<\frac{(1-\lambda_1)x}{\phi(q)}
\end{equation}
as required.
\section{Case 2: No Siegel Zeroes}\label{sec:NonSiegel}
We now consider the case where there are no $\eta$-Siegel zeros for some small fixed constant $\eta>0$. In this case we have $\lambda_\rho\ge \eta$ for all zeros $\rho$ with $|\Im(\rho)|\le q^2$. Following the method in the previous section and using this zero free region, we can establish Proposition \ref{Prpstn:OverallResult} if $\log{x}/\log{q}$ is sufficiently large. To obtain an explicit lower bound for the range of $\log{x}/\log{q}$ in which this holds, however, would require us to estimate the constant $C_3$ in Proposition \ref{Prpstn:ZeroDensity}, and would likely produce a very large bound if done directly.

We will follow the work done on the estimation of Linnik's constant to obtain an explicit lower bound for $\log{x}/\log{q}$ for which the result holds. We do this by we estimating weighted sums over primes and weighted zero density estimates in smaller regions. In particular, as in the case for estimating Linnik's constant, we specifically need sharp estimates for the zeros with real part close to 1. This section follows closely the method of Heath-Brown in \cite{HB}[Section 13].

We define the following quantities which we shall for the rest of the paper:
\begin{align}
M&:=\frac{\log{x}}{\log{q}},\\
\mathcal{L}&:=\log{q},\\
\phi_\chi&:=\begin{cases}
\frac{1}{4},\qquad &\text{$q$ cube-free or ord$(\chi)\le \log{q}$},\\
\frac{1}{3},&\text{otherwise},
\end{cases}\\
\mathcal{Z}(\chi)&:=\{\rho: L(\rho,\chi)=0\}.
\end{align}
\subsection{Weighted Sum over Primes}
We wish to investigate
\begin{equation}\psi(x;q,a)=\sum_{\substack{n\le x\\n\equiv a \pmod{q}}}\Lambda(n).\end{equation}
We fix a small positive constant $\epsilon>0$ and let
\begin{equation}f(t)=
\begin{cases}
0,&t\le 1/2\\
\frac{\log{x}}{\epsilon}(t-1/2),&1/2\le t \le 1/2+\epsilon/\log{x}\\
1,&1/2+\epsilon/\log{x}\le t\le 1\\
1-\frac{\log{x}}{\epsilon}(t-1),&1\le t \le 1+\epsilon/\log{x}\\
0,&1+\epsilon/\log{x}\le t.
\end{cases}
\end{equation}
The Brun-Titchmarsh theorem for primes in short intervals (see \cite{MontgomeryVaughan}, for example) states that
\begin{equation}\pi(x;q,a)-\pi(x-y;q,a)\le \frac{2y}{\phi(q)\log{y/q}}.\end{equation}
We replace the sum
\begin{equation}\sum_{\substack{n\le x\\ n \equiv a \pmod{q}}} \Lambda(n)\end{equation}
with the weighted sum 
\begin{equation}\sum_{\substack{n\le x\\n\equiv a \pmod{q}}}\Lambda(n)f\left(\frac{\log{n}}{\log{x}}\right).\end{equation}
By the Brun-Titchmarsh theorem for primes in short intervals and for $\epsilon$ sufficiently small, the error introduced by making this change is
\begin{align}
&\le \sum_{\substack{x\le n \le x e^{\epsilon}\\n\equiv a \pmod{q}}}\Lambda(n)+\sum_{n \le e^\epsilon x^{1/2}}\Lambda(n)\nonumber\\
&\le (\log{x e^\epsilon})(\pi(x e^\epsilon;q,a)-\pi(x;q,a))+e^\epsilon(\log{x})x^{1/2}\nonumber\\
&\le \frac{4\epsilon x}{\phi(q)}.
\end{align}
Thus in order to prove
\begin{equation}\left|\psi(x;q,a)-\frac{x}{\phi(q)}\right|\le \frac{(1-\epsilon)x}{\phi(q)},\end{equation}
it is sufficient to prove that
\begin{equation}\left|\sum_{\substack{n\le x\\n\equiv a \pmod{q}}}\Lambda(n)f\left(\frac{\log{n}}{\log{x}}\right)-\frac{x}{\phi(q)}\right|< \frac{(1-5\epsilon)x}{\phi(q)}.\end{equation}
We note also
\begin{equation}
\sum_{\substack{n\le x\\n\equiv a \pmod{q}}}\Lambda(n)f\left(\frac{\log{n}}{\log{x}}\right)=\frac{1}{\phi(q)}\sum_{\chi \pmod{q}}\overline{\chi(a)}\left(\sum_{n\leq x} \Lambda(n)f\left(\frac{\log{n}}{\log{x}}\right)\chi(n)\right).
\end{equation}
We now replace $\chi$ in the inner sum with the primitive character $\chi^*$ which induces it. This introduces an error
\begin{align}
&\ll \frac{1}{\phi(q)}\sum_{\chi}\sum_{p|q}\sum_{x^{1/2}\le p^e\le x}\log{p}\nonumber\\
&\ll \sum_{p|q}\log{x}\nonumber\\
&\ll q^{\epsilon}\log{x}\nonumber\\
&\le \epsilon x
\end{align}
(recalling that $x>q$).

Thus it is sufficient to prove that
\begin{equation}\label{weightedsum}
\left|\sum_{\chi}\overline{\chi}(a)\sum_{n=1}^{\infty}\Lambda(n)\chi^*(n)f\left(\frac{\log{n}}{\log{x}}\right)-x\right|\le (1-6\epsilon)x.
\end{equation}
\subsection{Sum over Zeroes}
We let $F$ be the Laplace transform of $f$. Hence
\begin{align}
F(s)&=\int_{0}^{\infty}\exp(-st)f(t)d t\nonumber\\
&=e^{-s}\left(\frac{1-\exp(s/2)}{-s}\right)\left(\frac{1-\exp(\frac{\epsilon}{\log{x}}s)}{-\frac{\epsilon}{\log{x}}s}\right)\exp\left(-\frac{\epsilon}{\log{x}}s\right).
\end{align}
From the Laplace inversion formula we have
\begin{equation}f\left(\frac{\log{n}}{\log{x}}\right)=\frac{\log{x}}{2\pi i}\int_{2-i \infty}^{2+i\infty}n^{-s}F(-s\log{x})d s.\end{equation}
Therefore for $\chi\ne\chi_0$ we have
\begin{align}
\sum_{n=1}^{\infty}\Lambda(n)\chi^*(n)f\left(\frac{\log{n}}{\log{x}}\right)&=\frac{\log{x}}{2\pi i}\int^{2+i\infty}_{2-i\infty}\left(-\frac{L'}{L}(s,\chi^*)\right)\left(F(-s\log{x})\right)d s\nonumber\\
&=-\log{x}\sum_{\rho}F(-\rho\log{x})\nonumber\\
&\qquad+\frac{\log{x}}{2\pi i}\int^{-1/2+i\infty}_{-1/2-i\infty}\left(-\frac{L'}{L}(s,\chi^*)\right)\left(F(-s\log{x})\right)d s
\end{align}
where $\sum_\rho$ indicates a sum over all non-trivial zeros of $L(s,\chi)$.

On $\Re s=-\frac{1}{2}$ we have
\begin{equation}\frac{L'}{L}(s,\chi^*)\ll \log(q(1+|s|)),\qquad F(-s\log{x})\ll x^{-1/4}|s|^{-2}(\log{x})^{-1}.\end{equation}
Hence, recalling that $q\le x$,
\begin{equation}\frac{\log{x}}{2\pi i}\int^{-1/2+i\infty}_{-1/2-i\infty}\left(-\frac{L'}{L}(s,\chi^*)\right)\left(F(-s\log{x})\right)d s=O(x^{-1/4}\log{x}).\end{equation}
Thus
\begin{align}
\sum_{\chi\ne\chi_0}\left|\sum_{n=1}^{\infty}\Lambda(n)\chi^*(n)f\left(\frac{\log{n}}{\log{x}}\right)\right|&\le \log{x}\sum_{\chi\ne\chi_0}\sum_{\rho}|F(-\rho\log{x})|+O(q x^{-1/4}\log{x})\nonumber\\
&\le \log{x}\sum_{\chi\ne\chi_0}\sum_{\rho}|F(-\rho\log{x})|+\epsilon x.\label{eq:FirstNonChi0}
\end{align}
We now consider the case $\chi=\chi_0$. We note that $\chi_0^*$ is identically 1. Hence by the prime number theorem we have
\begin{equation}\label{eq:FirstChi0}
\left|\sum_{n=1}^{\infty}\Lambda(n)\chi_0^*(n)f\left(\frac{\log{n}}{\log{x}}\right)-x\right|\le 3\epsilon x.
\end{equation}
Thus putting together \eqref{eq:FirstNonChi0} and \eqref{eq:FirstChi0} we have
\begin{align}
&\left|\sum_\chi\overline{\chi}(a)\sum_{n=1}^{\infty}\Lambda(n)\chi^*(n)f\left(\frac{\log{n}}{\log{x}}\right)-x\right|\nonumber\\
&\qquad\le \left|\sum_{n=1}^{\infty}\Lambda(n)\chi_0^*(n)f\left(\frac{\log{n}}{\log{x}}\right)-x\right|+\sum_{\chi\ne\chi_0}\left|\sum_{n=1}^{\infty}\Lambda(n)\chi^*(n)f\left(\frac{\log{n}}{\log{x}}\right)\right|\nonumber\\
&\qquad\le 4\epsilon x +\log{x}\sum_{\chi\ne \chi_0}\sum_{\rho}\left|F(-\rho\log{x})\right|.\label{eq:keypoint1}
\end{align}
In particular it is sufficient to prove that
\begin{equation}
\log{x}\sum_{\chi\ne \chi_0}\sum_{\rho}\left|F(-\rho\log{x})\right|\le (1-10\epsilon)x.
\end{equation}
We now consider the contribution from the other characters where $\chi\ne \chi_0$. We first consider all zeros $\rho=\beta+i\gamma$ of all $L$-functions $L(s,\chi)$ (with $\chi\ne\chi_0$) in the rectangle
\begin{equation}1-\frac{m+1}{\log{q}}\le\beta\le 1-\frac{m}{\log{q}},\qquad n\le |\gamma| \le 2n\end{equation}
for $n\ge 1$.

We use the well-known zero density estimate
\begin{equation}\sum_{\chi}N(\sigma,\chi,T)\ll q^{3(1-\sigma)}\left(1+T^{3/2}\right).\end{equation}
Thus there are
\begin{equation}\ll e^{3m}\left(1+n^{3/2}\right)\end{equation}
such zeros in the rectangle.

Each zero contributes
\begin{equation}\log{x}\left|F(-\rho\log{x})\right|\ll x\frac{\exp(-m\frac{\log{x}}{\log{q}})}{\epsilon n^2}\end{equation}
to the right hand side of \eqref{eq:keypoint1}.

Thus, provided $M>3$, there is a constant $R$ (depending only on $\epsilon$) such that the contribution of all zeros in the rectangles with $\max(m,n)\ge R$ is
\begin{equation}\le \epsilon x.\end{equation}
Similarly we consider zeros in the rectangle
\begin{equation}\max\left(\frac{1}{2},1-\frac{m+1}{\log{q}}\right)\le \beta\le 1-\frac{m}{\log{q}},\qquad |\gamma|\le 1.\end{equation}
There are
\begin{equation}\ll e^{3m}\end{equation}
such zeros, and each zero contributes
\begin{equation}\ll x\frac{\exp(-m\frac{\log{x}}{\log{q}})}{\epsilon}.\end{equation}

Therefore again provided $M>3$, the contribution from all zeros in rectangles with $m\ge R$ is $\le \epsilon x$.

We now consider the final rectangle
\begin{equation}0\le \beta \le \frac{1}{2},\qquad |\gamma|\le 1.\end{equation}
All zeros must have $\beta\ge q^{-1/2-1/100}$ for $q$ sufficiently large (by symmetry of zeros about the critical line and the non-existence of Siegel zeros which are within $q^{-1/2-1/100}$ of $1$).

There are
\begin{equation}\ll q^{3/2}\end{equation}
zeros in this rectangle, and each zero contributes
\begin{equation}\ll \frac{x^{1/2}q^{2/100}}{\epsilon}.\end{equation}
Therefore the contribution from these zeros is
\begin{align}
&\ll \frac{x^{1/2}q^{3/2+1/50}}{\epsilon}\nonumber\\
&\le \epsilon x.
\end{align}
Thus at a cost of $3\epsilon x$ we only need to consider the contribution of zeros $\rho$ satisfying
\begin{equation}|1-\Re(\rho)|\ll_{\epsilon} \frac{1}{\log{q}},\qquad \Im(\rho)\ll_\epsilon 1.\end{equation}
For such $\rho$, and for $\epsilon$ sufficiently small and $q$ sufficiently large, we have
\begin{equation}\left|\left(\frac{1-x^{-\rho/2}}{\rho}\right)e^{\epsilon\rho}\right|\le 1+3\epsilon.\end{equation}
Also, for any $z\in\mathbb{C}$ with $\Re(z)\ge 0$ we have
\begin{equation}\left|\frac{1-e^{-z}}{z}\right|\le 1.\end{equation}
Thus, putting $\Re(\rho)=1-\lambda_\rho/\log{q}$, and recalling that $q^{7.999}\le x$ we have
\begin{align}
\log{x}\left|F(-\rho\log{x})\right|&=x\exp(-(1-\rho)\log{x})\left|\left(\frac{1-x^{-\rho/2}}{\rho}\right)\left(\frac{1-e^{-\epsilon\rho}}{\epsilon\rho}\right)e^{\epsilon\rho}\right|\nonumber\\
&\le x\exp\left(-\lambda_\rho\frac{\log{x}}{\log{q}}\right)(1+3\epsilon)\nonumber\\
&= x\exp\left(-M\lambda_\rho\right)(1+3\epsilon).
\end{align}
As before, we have put
\begin{equation}
M=\frac{\log{x}}{\log{q}}.
\end{equation}
Thus we have shown that
\begin{equation}
\left|\psi(x;q,a)-\frac{x}{\psi(q)}\right|\le 12\epsilon x +(1+3\epsilon)x\sum_{\chi\ne \chi_0}\sum_{\rho}^*\exp(-M\lambda_\rho),
\end{equation}
where $\displaystyle\sum^*$ represents a sum over all zeros of $L(s,\chi)$ in 
\begin{equation}
\mathcal{R}=\left\{z:1-\frac{R}{\log{q}}\le \Re(z)\le 1,\Im(z)\le R\right\},\label{eq:Rdef}
\end{equation}
with $R$ a  constant (independent of $x$ and $q$).
\section{Zero Density Estimates}
We wish to estimate the sum
\[\sum_{\chi\ne\chi_0}\sum_{\rho\in \mathcal{R}\cap Z(\chi)}\exp(-M\lambda_\rho),\]
where
\[\mathcal{Z}(\chi):=\{\rho:L(\rho,\chi)=0\},\]
\[\mathcal{R}=\left\{z:1-\frac{R}{\log{q}}\le \Re(z)\le 1,\Im(z)\le R\right\}.\]
We do this by obtaining a zero density estimate for zeros in $\mathcal{R}$ by means of different weighted sums over zeros of $L(s,\chi)$. We note that by the log-free zero density estimate given in Proposition \ref{Prpstn:ZeroDensity} this sum is finite for any $M\in \mathbb{R}$. We specifically wish to show that the sum is $<1$ when $M=7.999$.

Similar sums have been looked at in the estimation of Linnik's constant. We will broadly follow the approach of Heath-Brown in \cite{HB}, but most of the estimates must be extended to cover a region where $\Im(\rho)\ll 1$ instead of $\Im(\rho) \ll \mathcal{L}^{-1}$.

We split $\mathcal{R}$ vertically into smaller rectangles each with height $1/\mathcal{L}$. We put
\begin{equation}\mathcal{R}_m:=\left\{z:1-\frac{R}{\mathcal{L}}\le \Re(z) \le 1,\frac{m-1/2}{\mathcal{L}}\le|\Im(z)|\le \frac{m+1/2}{\mathcal{L}}\right\}.\end{equation}
We label our non-principle characters (mod $q$) as $\chi^{(1)},\chi^{(2)},\dots$ in some order. For each character $\chi^{(j)}$, and for each rectangle $\mathcal{R}_m$ for which $L(s,\chi^{(j)})$ has a zero in $\mathcal{R}_m$ we pick a zero of $L(s,\chi^{(j)})$ with greatest real part, which we label $\rho^{(j,m)}$.

We introduce the notation
\begin{equation}\rho^{(j,m)}=\beta^{(j,m)}+i\gamma^{(j,m)},\qquad 1-\beta^{(j,m)}=\frac{\lambda^{(j,m)}}{\log{q}},\qquad \gamma^{(j,m)}=\frac{\nu^{(j,m)}}{\log{q}}.\end{equation}
We also specifically label special zeros $\rho_1$, $\rho'_1$ and $\rho_2$. We let $\rho_1$ be a zero  of $\prod_{\chi}L(s,\chi)$ which is in $\mathcal{R}$ and has largest real part. We let $\chi_1$ be the corresponding character. We let $\rho_2$ be a zero  of $\prod_{\chi,\chi\ne \chi_1,\overline{\chi}_1}L(s,\chi)$ which is in $\mathcal{R}$ and has largest real part. We let $\rho_1'$ be a zero of $L(s,\chi_1)$ which is in $ \mathcal{R}$ and is not $\rho_1$ or $\overline{\rho}_1$ but otherwise has largest real part. If $\rho_1$ is not a simple zero we simply have $\rho_1'=\rho_1$. 

For simplicity we argue as if $\rho_1,\rho_1',\rho_2$ all exist. Our argument is simpler and stronger if any of these do not exist.

We now wish to estimate separately a weighted sum over rectangles  and a weighted sum over zeros in any such rectangle. Specifically we wish to prove the following three lemmas:
\begin{lmm}\label{lmm:FirstZeroDensity}
For any $\delta>0$ any $m\in \mathbb{Z}$ and any constant $K>0$ we have for $q>q_0(\delta)$ that
\[\sum_{\rho\in \mathcal{R}_m\cap Z(\chi^{(j)})}B_1(\lambda_\rho)\le C_1(\lambda^{(j,m)})\] 
where
\begin{align*}
B_1(\lambda)&=\frac{\left(1-\exp(-K\lambda)\right)^{2}}{\lambda^2+1/4},\\
C_1(\lambda)&=\frac{\phi_\chi(1-\exp(-2K\lambda))}{2\lambda}+\frac{2K\lambda-1+\exp(-2K\lambda)}{2\lambda^2}+\delta.
\end{align*}
\end{lmm}
\begin{lmm}\label{lmm:OldZeroDensity}
let $(\chi^{(i)})_{i\in I}$ be a set of characters (mod $q$). Then for any $\delta>0$ and $q>q_0(\delta)$ we have
\[\sum_{m\in \mathbb{Z},i\in\mathbb{I}}B_2(\lambda^{(j,m)})\le C_2\]
where
\begin{align*}
B_2(\lambda)&=\left(\frac{e^{2\lambda x_1}+e^{2\lambda x_0}}{x_1-x_0}+\frac{e^{2\lambda u_1}+e^{2\lambda u_0}}{u_1-u_0}\right)^{-1},\\
C_2&=\left(\frac{x_1+x_0-v-u_1}{2w(v-u_1)}\right)(1+G_2)+\delta,\\
G_2&\text{ is defined in \eqref{eq:G2def},}
\end{align*}
and $x_1,x_0,v,u_1,u_0,w$ are all constants $>0$ satisfying
\[x_1>x_0,\qquad x_0>v+w+1/3,\qquad v>u_1,\qquad u_1>u_0,\qquad u_0>2w+1/3.\]
In particular, we have
\[\sum_{j,m}\left(\frac{e^{3.243...\lambda^{(j,m)}}+e^{2.823...\lambda^{(j,m)}}}{0.21}+\frac{e^{1.238...\lambda^{(j,m)}}+e^{1.126...\lambda^{(j,m)}}}{0.056}\right)^{-1}\le 11.826...\]
\end{lmm}
\begin{lmm}\label{lmm:NewZeroDensity}
Let $0\le \lambda\le 2$ be such that
\[G(\lambda-\lambda_{11})>g(0)/6\qquad \text{and} \qquad (G(\lambda-\lambda_{11})-g(0)/6)^2>G(-\lambda_{11})g(0)/6.\]
Then for any $\delta>0$ and $q>q_0(\delta,g)$
\[\sum_{\substack{j,m\\ \lambda^{(j,m)}\le \lambda}}1\le \frac{G(-\lambda_{11})G_3}{(G(\lambda-\lambda_{11})-g(0)/6)^2-G(-\lambda_{11})g(0)/6}+\delta\]
Where 
$g:[0,\infty)\rightarrow \mathbb{R}$ satisfies Condition \ref{Cndtn:NewCondition1} and \ref{Cndtn:NewCondition2}  given on page \pageref{Cndtn:NewCondition1}.
$G$ is the Lamplace transform of $g$.
$G_3$ is defined in equation \eqref{eq:G3def}.

In particular, we have the bounds given by Table \ref{Table:NewZeroDensityTable} on page \pageref{Table:NewZeroDensityTable}.
\end{lmm}
We will now proceed to prove each of these Lemmas in turn.

We note here that we can easily ensure the $L$ given in \cite{HB}[Lemma 6.1] satisfies $R\le L\le \frac{1}{10}\mathcal{L}$ rather than just $L\le \frac{1}{10}\mathcal{L}$ by following exactly the same argument but with this restriction. This means that all the results of Heath-Brown \cite{HB} and Xylouris \cite{Xylouris} which consider zeros in the region
\begin{equation}1-\frac{\log\log{\mathcal{L}}}{3\mathcal{L}}\le \sigma \le 1,\qquad |t|\le L\end{equation}
also apply to the zeros which we consider in $\mathcal{R}$.
\subsection{First Zero Density Estimate}
We now consider zeros within one of the rectangles $\mathcal{R}_m$. We follow almost identically the argument of Heath-Brown in \cite{HB}[Lemma 13.3].

We put
\begin{align}
h_1(t)&=
\begin{cases}
\sinh(K-t)\lambda,\qquad&0\le t\le K\\
0,&t\ge K,
\end{cases}\\
H_1(z)&=\int_0^\infty e^{-zt}h_1(t)dt=\frac{1}{2}\left(\frac{e^{K\lambda}}{\lambda+z}+\frac{e^{-K\lambda}}{\lambda-z}-\frac{2\lambda e^{-Kz}}{\lambda^2-z^2}\right),\\
H_2(z)&=\left(\frac{1-e^{-Kz}}{z}\right)^2,
\end{align}
for some constants $K\in \mathbb{R}$ and $\lambda\in\mathbb{C}$, which will be declared later.

We note that
\begin{equation}\Re(H_1(it))=\frac{\lambda e^{K\lambda}}{2}\left|\frac{1-e^{-K(\lambda+it)}}{\lambda+it}\right|^2=\frac{\lambda e^{K\lambda}}{2}\left|H_2(\lambda+it)\right|.\end{equation}
Since $H_1(z)$ and $H_2(\lambda+z)$ tend uniformly to zero in $\Re(z)\ge 0$ as $|z|\rightarrow \infty$, and $\Re(H_1(z))\ge \lambda e^{K\lambda}|H_2(\lambda+z)|/2$ when $\Re(z)=0$, by \cite{HB}[Lemma 4.1] we have
\begin{equation}
\Re(H_1(z))\ge \frac{\lambda e^{K\lambda}}{2}|H_2(\lambda+z)|
\end{equation}
whenever $\Re(z)\ge 0$.

We fix a character $\chi=\chi^{(j)}\ne \chi_0$ and take $\lambda=\lambda^{(j,k)}$. Therefore $L(s,\chi)$ has no zeros in the region $\{\sigma> 1-\lambda/\mathcal{L}\}\cap \mathcal{R}_m$.

Thus we have
\begin{equation}\sum_{\rho\in \mathcal{R}_m\cap Z(\chi)}|H_2((1-\rho+im/\mathcal{L})\mathcal{L})|\le \frac{2e^{-K\lambda}}{\lambda}\sum_{\rho\in \mathcal{R}_m\cap Z(\chi)}\Re(H_1((s-\rho)\mathcal{L})),\end{equation}
where $s=1-\lambda/\mathcal{L}+im/\mathcal{L}$.

By \cite{HB}[Lemma 5.2] and \cite{HB}[Lemma 5.3] we have (recalling that $|m|\ll \mathcal{L}$ so $|\Im(s)|\le \mathcal{L}$ for $q$ sufficiently large), for any given $\delta>0$ and $q>q(\delta)$
\begin{align}
\sum_{\rho\in \mathcal{R}_m\cap Z(\chi)}\Re(H_1((s-\rho)\mathcal{L}))&\le \frac{h_1(0)\phi_\chi}{2}+\mathcal{L}^{-1}\left|\sum_{n=1}^{\infty}\Lambda(n)\Re\left(\frac{\chi(n)}{n^s}\right)h_1(\mathcal{L}^{-1}\log{n})\right|+\delta\nonumber\\
&\le \frac{h_1(0)\phi_\chi}{2}+\mathcal{L}^{-1}\sum_{n=1}^{\infty}\Lambda(n)\left(\frac{\chi_0(n)}{n^{\Re(s)}}\right)h_1(\mathcal{L}^{-1}\log{n})+\delta\nonumber\\
&\le \frac{h_1(0)\phi_\chi}{2}+|H_1((\Re(s)-1)\mathcal{L})|+2\delta.
\end{align}
This gives
\begin{equation}\sum_{\rho\in \mathcal{R}_m\cap Z(\chi)}\left|\frac{1-e^{-K\lambda_\rho-iK(m-\gamma_\rho \mathcal{L})}}{\lambda_\rho+i(m-\gamma_\rho \mathcal{L})}\right|^2\le \frac{\phi_\chi(1-e^{-2K\lambda})}{2\lambda}+\frac{2K\lambda-1+e^{-2K\lambda}}{2\lambda^2}+2\delta.\end{equation}
Since $\rho\in \mathcal{R}_m$ we have $|m-\gamma_\rho \mathcal{L}|\le 1/2$. Thus, recalling that $\chi=\chi^{(j)}$ and $\lambda=\lambda^{(j,m)}$, we have
\begin{equation}\sum_{\rho\in \mathcal{R}_m\cap Z(\chi^{(j)})}\frac{(1-e^{-K\lambda_\rho})^2}{\lambda_\rho^2+1/4}\le \frac{\phi_\chi(1-e^{-2K\lambda^{(j,m)}})}{2\lambda^{(j,m)}}+\frac{2K\lambda^{(j,m)}-1+e^{-2K\lambda^{(j,m)}}}{2(\lambda^{(j,m)})^2}+2\delta.\end{equation}
Hence Lemma \ref{lmm:FirstZeroDensity} holds.
\subsection{Second Zero Density Estimate}
We now prove Lemma \ref{lmm:OldZeroDensity}. The proof uses ideas originally due to Graham \cite{Graham1}. We follow the method of \cite{HB}[Section 11], but extend the result to a weighted sum over zeros rather than just characters. We do this by using integrated exponential weights instead of exponential weights, an idea originally due to Jutila \cite{Jutila}.

We adopt similar notation to that of \cite{HB}[Section 11]. We put
\begin{equation}U_0=q^{u_0}, U_1=q^{u_1}, X_0=q^{x_0}, X_1=q^{x_1}, V=q^v, W=q^w\end{equation}
with constant exponents $0<w<u_0<u_1<v<x_0<x_1$ to be declared later. We put
\begin{equation}U=q^{u},X=q^{x}\end{equation}
with $u_0\le u\le u_1$ and $x_0\le x\le x_1$ parameters which we will integrate over.

We define
\begin{equation}\label{psi}
\psi_d=
\begin{cases}
\mu(d),\qquad&1\le d \le U_1\\
\mu(d)\frac{\log{V/d}}{\log{V/U_1}},&U_1\le d\le V\\
0,&d\ge V,
\end{cases}
\end{equation}
and
\begin{equation}\label{phi}
\theta_d=
\begin{cases}
\mu(d)\frac{\log{W/d}}{\log{W}},&1\le d\le W\\
0,&d\ge W.
\end{cases}
\end{equation}
We wish to study the sum
\begin{equation}J(\rho^{(j,m)},\chi):=w_{j,m}\sum_{n=1}^\infty\left(\sum_{d|n}\psi_d\right)\left(\sum_{d|n}\theta_d\right)\chi(n)n^{-\rho^{(j,m)}}j(n),\end{equation}
where
\begin{equation}j(n)=\left(\frac{\int_{x_0}^{x_1}\int_{u_0}^{u_1}(e^{-n/X}-e^{n\mathcal{L}^2/U})dudx}{(u_1-u_0)(x_1-x_0)}\right)\end{equation}
and $w_{j,m}$ are some non-negative weights.

We start with the following weighted-sum result.
\begin{lmm}
For $x_0>u_1+v+\phi_{\chi^{(j)}}$ we have:
\begin{equation*}\label{setup}
w_{j,m}^2\le (1+O(\mathcal{L}^{-1}))\left|J(\rho^{(j,m)},\chi^{(j)})\right|^2
\end{equation*}
\end{lmm}
\begin{proof}
The argument of \cite{HB}[Pages 317-318] shows that
\begin{equation}1=(1+O(\mathcal{L}^{-1}))\sum_{n=1}^{\infty}\left(\sum_{d|n}\psi_d\right)\left(\sum_{d|n}\theta_d\right)\chi^{(j)}(n)n^{-\rho^{(j,m)}}\left(e^{-n/X}-e^{-n\mathcal{L}^2/U}\right).\end{equation}
We note that in \cite{HB} the definition of $\psi_d$ is slightly different (it defined with constants labelled $U$ and $V$ rather than $U_1$ and $V$ as in our case), but this does not affect the argument in any way since $U_1\ge U$.

Multiplying the above expression by weights $w_{j,m}$ and integrating over $x\in[x_0,x_1]$ and $u\in[u_0,u_1]$ gives
\begin{equation}w_{j,m}=(1+O(\mathcal{L}^{-1}))J(\rho^{(j,m)},\chi^{(j)}).\end{equation}
Squaring both sides of the above expression gives the result.
\end{proof}
We sum the expression \eqref{setup} over all zeros $\rho^{(j,m)}$. We let $\smash{\displaystyle\sum_{j,m}}$ denote this sum.

Thus
\begin{equation}\sum_{j,m}w_{j,m}^2\le (1+O(\mathcal{L}^{-1}))\sum_{j,m}\left|J(\rho^{(j,m)},\chi^{(j)})\right|^2.\end{equation}
We now use the well-known duality principle, which we will state here for convenience.
\begin{lmm}[Duality Principle]\label{lmm:DualityPrinciple}
If
\begin{equation*}\sum_{n}\left|\sum_{j,m}a_{n,j,m}C_{j,m}\right|^2\le B\sum_{j,m}\left|C_{j,m}\right|^2\end{equation*}
for all choices of the coefficients $C_{j,m}$, then
\begin{equation*}\sum_{j,m}\left|\sum_{n}a_{n,j,m}b_{n}\right|^2\le B\sum_{n}\left| b_n\right|^2\end{equation*}
for any choice of $b_n$.
\end{lmm}
We wish to use Lemma \ref{lmm:DualityPrinciple} with
\begin{align}
a_{n,j,m}&=w_{j,m}\chi^{(j)}(n)n^{1/2-\rho^{(j,m)}}\left(\sum_{d|n}\theta_d\right)j(n)^{1/2},\\
b_n&=\left(\sum_{d|n}\psi_d\right)n^{-1/2}j(n)^{1/2},
\end{align}
to bound this sum. We note that
\begin{equation}\sum_{n=1}^{\infty}a_{n,j,m}b_n=J(\rho^{(j,m)},\chi^{(j)}).\end{equation}
First we evaluate $\sum b_n^2$.
\begin{lmm}\label{lmm:bnSum}
For $x_0>v$ we have:
\begin{equation*}\sum_{n=1}^\infty |b_n|^2=(1+O(\mathcal{L}^{-1}\log{\mathcal{L}}))\frac{x_1+x_0-u_1-v}{2(v-u_1)}.
\end{equation*}
\end{lmm}
\begin{proof}
The argument leading to equation (11.14) of \cite{HB}[Page 319] shows (recalling our definition of $\psi_d$ used parameters $U_1$ and $V$ rather than $U$ and $V$) that provided $x>v$ we have
\begin{equation}\sum_{n=1}^{\infty}\left(\sum_{d|n}\psi_d\right)^2n^{-1}(e^{-n/X}-e^{-n\mathcal{L}^2/U})=(1+O(\mathcal{L}^{-1}\log{\mathcal{L}}))\frac{2x-u_1-v}{2(v-u_1)}.\end{equation}
Since we have $x\ge x_0>v$ this holds in our case.

Therefore, integrating with respect to $x\in[x_0,x_1]$ and $u\in[u_0,u_1]$ and dividing through by $(x_1-x_0)(u_1-u_0)$ gives
\begin{align}\sum_{n=1}^{\infty}\left(\sum_{d|n}\psi_d\right)^2n^{-1}&\frac{\int_{x_0}^{x_1}\int_{u_0}^{u_1}(e^{-n/X}-e^{-n\mathcal{L}^2/U})dudx}{(x_1-x_0)(u_1-u_0)}\nonumber\\
&=(1+O(\mathcal{L}^{-1}\log{\mathcal{L}}))\frac{x_1+x_0-u_1-v}{2(v-u_1)}.\end{align}
Hence the result holds.
\end{proof}
Therefore in order to use Lemma \ref{lmm:DualityPrinciple} we want to find a bound $B$ such that
\begin{equation}\label{duality}
\sum_{n=1}^{\infty}\left|\sum_{j,m}a_{n,j,m}C_{j,m}\right|^2\le B\sum_{j,m}\left|C_{j,m}\right|^2
\end{equation}
for any possible choice of $C_{j,m}$.

Expanding the left hand side, terms are of the form
\begin{align}
\sum_{n=1}^\infty &a_{n,j_1,m_1}a_{n,j_2,m_2}C_{j_1,m_1}\overline{C}_{j_2,m_2}\nonumber\\
&=
C_{j_1,m_1}\overline{C}_{j_2,m_2}w_{j_1,m_1}w_{j_2,m_2}\sum_{n=1}^{\infty}\left(\sum_{d|n}\theta_d\right)^2\chi^{(j_1)}(n)\overline{\chi}^{(j_2)}(n)n^{1-\rho^{(j_1,m_1)}-\overline{\rho}^{(j_2,m_2)}}j(n).
\end{align}
To ease notation we let
\begin{equation}\rho_{(1)}=\rho^{(j_1,m_1)},\qquad \rho_{(2)}=\rho^{(j_2,m_2)},\end{equation}
and correspondingly define $\chi_{(1)}, \chi_{(2)},\beta_{(1)}, \beta_{(2)},\lambda_{(1)},\lambda_{(2)},\gamma_{(1)},\gamma_{(2)}$.

We first deal with the terms when $\chi_{(1)}\ne \chi_{(2)}$.

We put
\begin{equation}J_2(s,\chi)=\sum_{w_1,w_2\le W}\theta_{w_1}\theta_{w_2}\chi([w_1,w_2])[w_1,w_2]^{-s}.\end{equation}
(Here $[a,b]$ denotes the least common multiple of $a$ and $b$).

By the inverse Laplace transform of the exponential function we have
\begin{align}
\sum_{n=1}^{\infty}&\left(\sum_{d|n}\theta_d\right)^2\chi_{(1)}(n)\overline{\chi}_{(2)}(n)n^{1-\rho_{(1)}-\overline{\rho}_{(2)}}(e^{-n/X}-e^{-n\mathcal{L}^2/U})\nonumber\\
&=\frac{1}{2\pi i}\int^{1+i\infty}_{1-i\infty}L(s+\rho_{(1)}+\overline{\rho}_{(2)}-1,\chi_{(1)}\overline{\chi}_{(2)})(X^s-(U\mathcal{L}^{-2})^s)\nonumber\\
&\qquad\qquad\qquad\qquad\qquad\times \Gamma(s)J_2(s+\rho_{(1)}+\overline{\rho}_{(2)}-1,\chi_{(1)}\overline{\chi}_{(2)})d s\nonumber\\
&=\frac{1}{2\pi i}\int^{2-\beta_{(1)}-\beta_{(2)}-1/k+i\infty}_{2-\beta_{(1)}-\beta_{(2)}-1/k-i\infty}L(s+\rho_{(1)}+\overline{\rho}_{(2)}-1,\chi_{(1)}\overline{\chi}_{(2)})(X^s-(U\mathcal{L}^{-2})^s)\nonumber\\
&\qquad\qquad\qquad\qquad\qquad\times \Gamma(s)J_2(s+\rho_{(1)}+\overline{\rho}_{(2)}-1,\chi_{(1)}\overline{\chi}_{(2)})ds.
\end{align}
where $k>10$ is a fixed constant (to be declared later).

On $\Re(s)=2-\beta_{(1)}-\beta_{(2)}-1/k$ with $\chi\ne\chi_0$ we have
\begin{align}L(s+\rho_{(1)}+\overline{\rho}_{(2)}-1,\chi)&\ll_k q^{\phi_\chi/k+1/k^2}(1+|t|),\\
\Gamma(s)&\ll e^{-|t|},\\
J_2(s+\rho_{(1)}+\overline{\rho}_{(2)}-1,\chi)&\ll \sum [w_1,w_2]^{-1+1/k}\nonumber\\
&\ll \sum_{n\le W^2}n^{-1+1/k}d(n)^2\nonumber\\
&\ll W^{2/k}\mathcal{L}^3.\end{align}
Thus, letting $\chi=\chi_{(1)}\overline{\chi}_{(2)}$, we obtain
\begin{align}
\frac{1}{2\pi i}\int_{1-\beta_{(1)}-\beta_{(2)}-1/k-i\infty}^{1-\beta_{(1)}-\beta_{(2)}-1/k+i\infty}&L(s+\rho_{(1)}+\overline{\rho}_{(2)}-1,\chi)\Gamma(s)(X^s-(U\mathcal{L}^{-2})^s)\nonumber\\
&\times J_2(s+\rho_{(1)}+\overline{\rho}_{(2)}-1,\chi)d s\nonumber\\
&\qquad\qquad\ll (q^{\phi_\chi} W^2U^{-1}\mathcal{L}^3)^{1/k}q^{1/k^2}\mathcal{L}^2(U\mathcal{L}^{-2})^{2-\beta_{(1)}-\beta_{(2)}}\nonumber\\
&\qquad\qquad\ll (q^{\phi_\chi} W^2U^{-1})^{1/k}q^{2/k^2}.
\end{align}
(Recalling $1-\beta_{(1)}$ and $1-\beta_{(2)}$ are $o(1)$)

This is $O(\mathcal{L}^{-1})$ provided that $k$ is chosen sufficiently large and (recalling $\phi_\chi\le 1/3$ for all $\chi$) provided we have that
\begin{equation}\label{cond2}
u_0>2w+1/3.
\end{equation}
The terms with $\chi_{(1)}\ne \chi_{(2)}$ therefore contribute
\begin{equation}\label{nondiagonal}
\ll \mathcal{L}^{-1}\left(\sum_{j,m}|C_{j,m}|w_{j,m}\right)^2\ll \mathcal{L}^{-1}\left(\sum_{j,m}w_{j,m}^2\right)\sum_{j,m}|C_{j,m}|^2.
\end{equation}
We now consider terms with $\chi_{(1)}=\chi_{(2)}$. Such terms are of the form
\begin{equation}\label{terms}
C_{(1)}\overline{C}_{(2)}(w_{(1)}w_{(2)})\sum_{n=1}^{\infty}\left(\sum_{d|n}\theta_d\right)^2\chi_0(n)n^{1-\rho_{(1)}-\overline{\rho}_{(2)}}j(n).
\end{equation}
\begin{lmm}
For $x>v$ we have:
\begin{align*}
&\left|\sum_{n=1}^{\infty}\left(\sum_{d|n}\theta_d\right)^2\chi_0(n)n^{1-\rho_{(1)}-\overline{\rho}_{(2)}}j(n)\right|\le\left|\frac{(1+O(\mathcal{L}^{-1}\log{\mathcal{L}}))}{w\mathcal{L}^2(2-\rho_{(1)}-\overline{\rho}_{(2)})^2}\right|\\
&\qquad\qquad\times\left|\frac{X_1^{2-\rho_{(1)}-\overline{\rho}_{(2)}}-X_0^{2-\rho_{(1)}-\overline{\rho_{(2)}}}}{x_1-x_0}-\frac{U_1^{2-\rho_{(1)}-\overline{\rho}_{(2)}}-U_0^{2-\rho_{(1)}-\overline{\rho_{(2)}}}}{u_1-u_0}\right|+O(\mathcal{L}^{-1}).
\end{align*}
\end{lmm}
\begin{proof}
We have that
\begin{align}
\sum_{n=1}^\infty \left(\sum_{d|n}\theta_d\right)^2&\chi_0(n)n^{1-\rho_{(1)}-\overline{\rho}_{(2)}}\left(e^{-k[d_1,d_2]/X}-e^{k[d_1,d_2]\mathcal{L}^2/U}\right)\nonumber\\
&=\sum_{d_1,d_2}\theta_{d_1}\theta_{d_2}\chi_0([d_1,d_2])[d_1,d_2]^{1-\rho_{(1)}-\overline{\rho}_{(2)}}\nonumber\\
&\qquad\qquad\times\sum_{k=1}^\infty k^{1-\rho_{(1)}-\overline{\rho}_2}\chi_0(k)\left(e^{-k[d_1,d_2]/X}-e^{k[d_1,d_2]\mathcal{L}^2/U}\right).
\end{align}
By the inverse Laplace transform of the exponential function again we have
\begin{align}
&\sum_{k=1}^{\infty}\chi_0(k)k^{1-\rho_{(1)}-\overline{\rho}_{(2)}}(e^{-k[d_1,d_2]/X}-e^{-k[d_1,d_2]\mathcal{L}^2/U})\nonumber\\
&=\frac{1}{2\pi i}\int^{1+i\infty}_{1-i\infty}L(s+\rho_{(1)}+\overline{\rho}_{(2)}-1,\chi_0)\Gamma(s)\left(\left(\frac{X}{[d_1,d_2]}\right)^s-\left(\frac{U}{\mathcal{L}^2[d_1,d_2]}\right)^s\right)ds.
\end{align}
We again move the line of integration to $\Re(s)=2-\beta_{(1)}-\beta_{(2)}-1/k$, and by exactly the same reasoning, we have that the integral over this contour is negligible when $u_0>2w$. We encounter a pole at $s=2-\rho_{(1)}-\overline{\rho}_{(2)}$, however, which contributes
\begin{equation}\frac{\phi(q)}{q}\Gamma(2-\rho_{(1)}-\overline{\rho}_{(2)})\left(\left(\frac{X}{[d_1,d_2]}\right)^{2-\rho_{(1)}-\overline{\rho}_{(2)}}-\left(\frac{U}{\mathcal{L}^{2}[d_1,d_2]}\right)^{2-\rho_{(1)}-\overline{\rho}_{(2)}}\right).\end{equation}
Thus
\begin{align}
\sum_{n=1}^\infty& \left(\sum_{d|n}\theta_d\right)^2\chi_0(n)n^{1-\rho_{(1)}-\overline{\rho}_{(2)}}\left(e^{-k[d_1,d_2]/X}-e^{k[d_1,d_2]\mathcal{L}^2/U}\right)\nonumber\\
&=\frac{\phi(q)}{q}\Gamma(2-\rho_{(1)}-\overline{\rho}_{(2)})\left(X^{2-\rho_{(1)}-\overline{\rho}_{(2)}}-(U\mathcal{L}^{-2})^{2-\rho_{(1)}-\overline{\rho}_{(2)}}\right)\sum_{d_1,d_2}\frac{\theta_{d_1}\theta_{d_2}\chi_0([d_1,d_2])}{[d_1,d_2]}\nonumber\\
&\qquad\qquad+O(\mathcal{L}^{-1}).
\end{align}
We now perform the integrations with respect to $x$ and $u$. We have
\begin{align}
&\frac{1}{(x_1-x_0)(u_1-u_0)}\int_{x_0}^{x_1}\int_{u_0}^{u_1}\left(X^{2-\rho_{(1)}-\overline{\rho}_{(2)}}-(U\mathcal{L}^{-2})^{2-\rho_{(1)}-\overline{\rho}_{(2)}}\right)du d x\nonumber\\
&=\frac{1}{\mathcal{L}(2-\rho_{(1)}-\overline{\rho}_{(2)})}\left(\frac{X_1^{2-\rho_{(1)}-\overline{\rho}_{(2)}}-X_0^{2-\rho_{(1)}-\overline{\rho_{(2)}}}}{x_1-x_0}-\frac{U_1^{2-\rho_{(1)}-\overline{\rho}_{(2)}}-U_0^{2-\rho_{(1)}-\overline{\rho_{(2)}}}}{u_1-u_0}\right).
\end{align}
Thus
\begin{align}
&\left|\sum_{n=1}^\infty \left(\sum_{d|n}\theta_d\right)^2\chi_0(n)n^{1-\rho_{(1)}-\overline{\rho}_{(2)}}j(n)\right|\nonumber\\
&\le\frac{\phi(q)}{\mathcal{L}q}\left|\frac{\Gamma(2-\rho_{(1)}-\overline{\rho}_{(2)})}{2-\rho_{(1)}-\overline{\rho}_{(2)}}\right|\times\left|\sum_{d_1,d_2}\frac{\theta_{d_1}\theta_{d_2}\chi_0([d_1,d_2])}{[d_1,d_2]}\right|\nonumber\\
&\qquad\qquad\times\left|\frac{X_1^{2-\rho_{(1)}-\overline{\rho}_{(2)}}-X_0^{2-\rho_{(1)}-\overline{\rho_{(2)}}}}{x_1-x_0}-\frac{U_1^{2-\rho_{(1)}-\overline{\rho}_{(2)}}-U_0^{2-\rho_{(1)}-\overline{\rho_{(2)}}}}{u_1-u_0}\right|+O(\mathcal{L}^{-1}).
\end{align}
We now estimate the sum over $d_1,d_2$. We have
\begin{align}
\left|\sum_{d_1,d_2}\theta_{d_1}\theta_{d_2}[d_1,d_2]^{-1}\chi_0([d_1,d_2])\right|&=\frac{1}{N}\Biggl|\sum_{d_1,d_2\le W}\theta_{d_1}\theta_{d_2}\Biggl(\frac{q}{\phi(q)}\sum_{\substack{[d_1,d_2]|n\\n\le N\\(n,q)=1}}1+O(q)\Biggr)\Biggr|\nonumber\\
&=\frac{q}{\phi(q)N}\sum_{\substack{n\le N\\(n,q)=1}}\left(\sum_{d|n}\theta_d\right)^2+O(qW^2N^{-1})\nonumber\\
&\le \frac{q}{\phi(q)N}\sum_{n\le N}\left(\sum_{d|n}\theta_d\right)^2+O(qW^2N^{-1}).
\end{align}
Graham \cite{Graham2} has shown that for $N>q^2W^2$ we have
\begin{equation}N^{-1}\sum_{n\le N}\left(\sum_{d|n}\theta_d\right)^2=\frac{1+O(\mathcal{L}^{-1})}{\log{W}}.\end{equation}
Hence for $N>q^2W^2$ we have
\begin{align}
\left|\sum_{d_1,d_2}\theta_{d_1}\theta_{d_2}[d_1,d_2]^{-1}\chi_0([d_1,d_2])\right|&\le\frac{q}{\phi(q)N}\sum_{n\le N}\left(\sum_{d|n}\theta_d\right)^2+O\left(q^{-1}\right)\nonumber\\
&=\frac{(1+O(\mathcal{L}^{-1}))q}{\phi(q)\log{W}}\nonumber\\
&=(1+O(\mathcal{L}^{-1}))\frac{q}{\phi(q)w\mathcal{L}}.
\end{align}
Thus
\begin{align}
&\left|\sum_{n=1}^\infty \left(\sum_{d|n}\theta_d\right)^2\chi_0(n)n^{1-\rho_{(1)}-\overline{\rho}_{(2)}}j(n)\right|\le\frac{(1+O(\mathcal{L}^{-1}))}{\mathcal{L}^2w}\left|\frac{\Gamma(2-\rho_{(1)}-\overline{\rho}_{(2)})}{2-\rho_{(1)}-\overline{\rho}_{(2)}}\right|\nonumber\\
&\qquad\qquad\times\left|\frac{X_1^{2-\rho_{(1)}-\overline{\rho}_{(2)}}-X_0^{2-\rho_{(1)}-\overline{\rho_{(2)}}}}{x_1-x_0}-\frac{U_1^{2-\rho_{(1)}-\overline{\rho}_{(2)}}-U_0^{2-\rho_{(1)}-\overline{\rho_{(2)}}}}{u_1-u_0}\right|+O(\mathcal{L}^{-1}).
\end{align}
We recall the Weierstrass product expansion of $\Gamma(s)$
\begin{equation}\Gamma(s)=\frac{e^{-\gamma s}}{s}\prod_{n=1}^{\infty}\left(1+\frac{s}{n}\right)^{-1}e^{s/n}.\end{equation}
We see that when $s=2-\rho_{(1)}-\overline{\rho}_{(2)}$, since $2-\beta_{(1)}-\beta_{(2)}=O(\mathcal{L}^{-1}\log{\mathcal{L}})$, we have
\begin{align}
\left|\Gamma(s)\right|&\le \frac{e^{-\gamma\Re(s)}}{|s|}\prod_{n=1}^\infty\left|1+\frac{s}{n}\right|^{-1}e^{\Re(s)/n}\nonumber\\
&\le \frac{1+O(\mathcal{L}^{-1}\log{\mathcal{L}})}{|s|}\prod_{n=1}^\infty\left(1+\frac{\Re(s)}{n}\right)^{-1}e^{\Re(s)/n}\nonumber\\
&\le \frac{1+O(\mathcal{L}^{-1}\log{\mathcal{L}})}{|2-\rho_{(1)}-\overline{\rho}_{(2)}|}\prod_{n=1}^\infty\left(1+O\left(\frac{\Re(s)}{n^2}\right)\right)\nonumber\\
&\le\frac{1+O(\mathcal{L}^{-1}\log{\mathcal{L}})}{|2-\rho_{(1)}-\overline{\rho}_{(2)}|}.
\end{align}
This completes the proof.
\end{proof}
To simplify notation we put
\begin{equation}j_2(\rho_{(1)},\rho_{(2)})=\frac{1}{\mathcal{L}^2(2-\rho_{(1)}-\overline{\rho}_{(2)})^2}\left(\frac{X_1^{2-\rho_{(1)}-\overline{\rho}_{(2)}}-X_0^{2-\rho_{(1)}-\overline{\rho_{(2)}}}}{x_1-x_0}-\frac{U_1^{2-\rho_{(1)}-\overline{\rho}_{(2)}}-U_0^{2-\rho_{(1)}-\overline{\rho_{(2)}}}}{u_1-u_0}\right).\end{equation}
Thus the sum over all the terms of the form \eqref{terms} with $\chi_{(1)}=\chi_{(2)}$ is
\begin{align}\le \frac{(1+O(\mathcal{L}^{-1}\log{\mathcal{L}}))}{w}\sum_{\substack{\rho_{(1)},\rho_{(2)}\\ \chi_{(1)}=\chi_{(2)}}}\left|C_{(1)}C_{(2)}w_{(1)}w_{(2)}j_2(\rho_{(1)},\rho_{(2)})\right|\nonumber\\
+O(\mathcal{L}^{-1}\sum_{\substack{\rho_{(1)},\rho_{(2)}\\ \chi_{(1)}=\chi_{(2)}}}\left| C_{(1)}C_{(2)}w_{(1)}w_{(2)})\right|.\end{align}
We put
\begin{equation}
G_2=\max_{\rho_{(1)}}\sum_{\substack{\rho_{(2)}\\ \chi_{(2)}=\chi_{(1)}}}\left|w_{(1)}w_{(2)}j_2(\rho_{(1)},\rho_{(2)})\right|.\label{eq:G2def}
\end{equation}
Hence
\begin{equation}\label{diagonal}
\sum_{\substack{\rho_{(1)},\rho_{(2)}\\ \chi_{(1)}=\chi_{(2)}}}\left|C_{(1)}C_{(2)}w_{(1)}w_{(2)}j_2(\rho_{(1)},\rho_{(2)})\right|\le G_2\sum_{\rho_{(1)}}|C_{(1)}|^2.
\end{equation}
Combining \eqref{nondiagonal} and \eqref{diagonal} we have
\begin{equation}\sum_{n=1}^{\infty}\left|\sum_{j,m}a_{n,j,m}C_{j,m}\right|^2\le \left(\frac{G_2}{w}\left(1+O(\mathcal{L}^{-1}\log{\mathcal{L}})\right)+O\left((\mathcal{L}^{-1}\sum_{j,m}w_{j,m}^2)\right)\right)\sum_{j,m}\left|C_{j,m}\right|^2\end{equation}
for any choice of the coefficients $C_{j,m}$.

Therefore by Lemma \ref{lmm:DualityPrinciple} and Lemma \ref{lmm:bnSum} we have
\begin{equation}\sum_{j,m}w_{j,m}^2\le \left(1+O(\mathcal{L}^{-1}\log{\mathcal{L}})\right)\left(\frac{G_2}{w}+O\left(\mathcal{L}^{-1}\sum_{j,m}w_{j,m}^2\right)\right)\left(\frac{x_1+x_0-u_1-v}{2(v-u_1)}\right)\end{equation}
which gives
\begin{equation}
\sum_{j,m}w_{j,m}^2\le \left(1+O(\mathcal{L}^{-1}\log{\mathcal{L}})\right)\left(\frac{x_1+x_0-u_1-v}{2w(v-u_1)}\right)G_2.
\end{equation}
We are therefore left to choose suitable weights $w_{j,m}$, bound $G_2$ and choose suitable constants $w,u_0,u_1,v,x_0,x_1$.

We note that, using Cauchy's inequality, we have
\begin{align}
&\left|\frac{X_1^{2-\rho_{(1)}-\overline{\rho_{(2)}}}-X_0^{2-\rho_{(1)}-\overline{\rho}_{(2)}}}{x_1-x_0}-\frac{U_1^{2-\rho_{(1)}-\overline{\rho_{(2)}}}-U_0^{2-\rho_{(1)}-\overline{\rho}_{(2)}}}{u_1-u_0}\right|\nonumber\\
&\qquad\le\left(\frac{e^{(\lambda_{(1)}+\lambda_{(2)})x_1}+e^{(\lambda_{(1)}+\lambda_{(2)})x_0}}{x_1-x_0}+\frac{e^{(\lambda_{(1)}+\lambda_{(2)})u_1}+e^{(\lambda_{(1)}+\lambda_{(2)})u_0}}{u_1-u_0}\right)\nonumber\\
&\qquad\le\left(\frac{e^{2\lambda_{(1)}x_1}+e^{2\lambda_{(1)}x_0}}{x_1-x_0}+\frac{e^{2\lambda_{(1)}u_1}+e^{2\lambda_{(1)}u_0}}{u_1-u_0}\right)^{1/2}\nonumber\\
&\qquad\qquad\times\left(\frac{e^{2\lambda_{(2)}x_1}+e^{2\lambda_{(2)}x_0}}{x_1-x_0}+\frac{e^{2\lambda_{(2)}u_1}+e^{2\lambda_{(2)}u_0}}{u_1-u_0}\right)^{1/2}.
\end{align}
Also
\begin{align}
\sum_{\rho_{(2)}}\left|\mathcal{L}^{-2}(2-\rho_{(1)}-\rho_{(2)})^{-2}\right|&=\sum_{\rho_{(2)}}\frac{1}{(\lambda_{(1)}+\lambda_{(2)})^2+(v_{(1)}-v_{(2)})^2}\nonumber\\
&\le 2\sum_{m=0}^\infty\frac{1}{(\lambda_{(1)}+\lambda_{(2)})^2+m^2},
\end{align}
since $|\Im(\rho^{(j,m_1)})-\Im(\rho^{(j,m_2)})|\ge (|m_1-m_2|-1)/\mathcal{L}$ by our choice of the rectangles $\mathcal{R}_m$.

Motivated by these observations we choose
\begin{equation}w_{j,m}=\left(\frac{e^{2\lambda^{(j,m)}x_1}+e^{2\lambda^{(j,m)}x_0}}{x_1-x_0}+\frac{e^{2\lambda^{(j,m)}u_1}+e^{2\lambda^{(j,m)}u_0}}{u_1-u_0}\right)^{-1/2}.\end{equation}
We assume from here on that we are only considering zeros $\rho^{(j,m)}$ with $\lambda^{(j,m)}\ge \lambda_{min}$.

We now wish to estimate $G_2$, and so bound $\sum_{\rho_{(2)}}|w_{(1)}w_{(2)}j_2(\rho_{(1)},\rho_{(2)})|$. We assume $\rho_{(1)}$ is in a rectangle $\mathcal{R}_{m_1}$ and then consider the contributions $G_{2,c}$ from zeros in rectangles $\mathcal{R}_{m_2}$ where $|m_1-m_2|=c\in\mathbb{Z}$ (since we have picked a fixed zero in each rectangle, there are at most 2 zeros corresponding to each choice of $c$).

We fist consider $c=0$. In this case $\rho_{(2)}=\rho_{(1)}$ (and there is only one zero). This contributes at most
\begin{align} G_{2,0}&\le\sup_{\rho_{(1)}}\left|j(\rho_{(1)},\rho_{(1)})w_{(1)}^2\right|\nonumber\\
&=\sup_{\rho_{(1)}}\left(\frac{X_1^{2-2\beta_{(1)}}-X_0^{2-2\beta_{(1)}}}{x_1-x_0}-\frac{U_1^{2-2\beta_{(1)}}-U_0^{2-2\beta_{(1)}}}{u_1-u_0}\right)\nonumber\\
&\qquad\qquad\times\left(\frac{X_1^{2-2\beta_{(1)}}+X_0^{2-2\beta_{(1)}}}{x_1-x_0}+\frac{U_1^{2-2\beta_{(1)}}+U_0^{2-2\beta_{(1)}}}{u_1-u_0}\right)^{-1}(2\lambda_{(1)})^{-2}\nonumber\\
&=\sup_{\lambda_{(1)}\ge\lambda_{min}}\left(\frac{e^{2x_1\lambda_{(1)}}-e^{2x_0\lambda_{(1)}}}{x_1-x_0}-\frac{e^{2u_1\lambda_{(1)}}-e^{2u_0\lambda_{(1)}}}{u_1-u_0}\right)\nonumber\\
&\qquad\qquad\times\left(\frac{e^{2x_1\lambda_{(1)}}+e^{2x_0\lambda_{(1)}}}{x_1-x_0}+\frac{e^{2u_1\lambda_{(1)}}+e^{2u_0\lambda_{(1)}}}{u_1-u_0}\right)^{-1}(2\lambda_{(1)})^{-2}.
\end{align}
We now deal with zeros with $1\le c\le 6$. This means that $c-1\le |\Im(\rho_{(1)})-\Im(\rho_{(2)})|\le c+1$. Thus these zeros contribute at most
\begin{align}
&2\sup_{\substack{\lambda_{(1)},\lambda_{(2)}\ge\lambda_{min}\\ c-1\le t\le c+1}}\left|\frac{e^{x_1(\lambda_{(1)}+\lambda_{(2)}+it)}-e^{x_0(\lambda_{(1)}+\lambda_{(2)}+it)}}{x_1-x_0}-\frac{e^{u_1(\lambda_{(1)}+\lambda_{(2)}+it)}-e^{u_0(\lambda_{(1)}+\lambda_{(2)}+it)}}{u_1-u_0}\right|\nonumber\\
&\qquad\times \left(\frac{e^{2x_1\lambda_{(1)}}+e^{2x_0\lambda_{(1)}}}{x_1-x_0}+\frac{e^{2u_1\lambda_{(1)}}+e^{2u_0\lambda_{(1)}}}{u_1-u_0}\right)^{-1/2}\left((\lambda_{(1)}+\lambda_{(2)})^2+t^2\right)^{-1}\nonumber\\
&\qquad\times \left(\frac{e^{2x_1\lambda_{(2)}}+e^{2x_0\lambda_{(2)}}}{x_1-x_0}+\frac{e^{2u_1\lambda_{(2)}}+e^{2u_0\lambda_{(2)}}}{u_1-u_0}\right)^{-1/2}.
\end{align}
By Cauchy's inequality
\begin{equation}e^{2\lambda_{(1)}x_1+2\lambda_{(2)}x_0}+e^{2\lambda_{(1)}x_0+2\lambda_{(2)}x_1}\ge 2e^{(\lambda_{(1)}+\lambda_{(2)})(x_1+x_0)},\end{equation}
and so
\begin{equation}(e^{2\lambda_{(1)}x_1}+e^{2\lambda_{(1)}x_0})(e^{2\lambda_{(2)}x_1}+e^{2\lambda_{(2)}x_0})\ge (e^{(\lambda_{(1)}+\lambda_{(2)})x_1}+e^{(\lambda_{(1)}+\lambda_{(2)})x_0})^2.\end{equation}
Similarly
\begin{equation}(e^{2\lambda_{(1)}u_1}+e^{2\lambda_{(1)}u_0})(e^{2\lambda_{(2)}u_1}+e^{2\lambda_{(2)}u_0})\ge (e^{(\lambda_{(1)}+\lambda_{(2)})u_1}+e^{(\lambda_{(1)}+\lambda_{(2)})u_0})^2.\end{equation}
Using Cauchy's inequality again we have
\begin{equation}e^{2\lambda_{(1)}x_i}+e^{2\lambda_{(1)}u_j}\ge2e^{(\lambda_{(1)}+\lambda_{(2)})(x_i+u_j)}\end{equation}
for any $i,j\in \{0,1\}$. Summing over all such $i,j$ gives
\begin{align}
(e^{2\lambda_{(1)}x_1}+e^{2\lambda_{(1)}x_0})&(e^{2\lambda_{(2)}u_1}+e^{2\lambda_{(2)}u_0})+(e^{2\lambda_{(2)}x_1}+e^{2\lambda_{(2)}x_0})(e^{2\lambda_{(1)}u_1}+e^{2\lambda_{(1)}u_0})\nonumber\\
&\ge 2(e^{(\lambda_{(1)}+\lambda_{(2)})x_1}+e^{(\lambda_{(1)}+\lambda_{(2)})x_0})(e^{(\lambda_{(1)}+\lambda_{(2)})u_1}+e^{(\lambda_{(1)}+\lambda_{(2)})u_0}).
\end{align}
Putting these together gives
\begin{align}
\left(\frac{e^{2x_1\lambda_{(2)}}+e^{2x_0\lambda_{(2)}}}{x_1-x_0}+\frac{e^{2u_1\lambda_{(2)}}+e^{2u_0\lambda_{(2)}}}{u_1-u_0}\right)\left(\frac{e^{2x_1\lambda_{(1)}}+e^{2x_0\lambda_{(1)}}}{x_1-x_0}+\frac{e^{2u_1\lambda_{(1)}}+e^{2u_0\lambda_{(1)}}}{u_1-u_0}\right)\nonumber\\
\ge
\left(\frac{e^{(\lambda_{(1)}+\lambda_{(2)})x_1}+e^{(\lambda_{(1)}+\lambda_{(2)})x_0}}{x_1-x_0}+\frac{e^{(\lambda_{(1)}+\lambda_{(2)})u_1}+e^{(\lambda_{(1)}+\lambda_{(2)})u_0}}{u_1-u_0}\right)^2.
\end{align}
Hence
\begin{align}
G_{2,c}&\le 2\sup_{\substack{\lambda\ge\lambda_{min}\\ c-1\le t\le c+1}}\left|\frac{e^{x_1(2\lambda+it)}-e^{x_0(2\lambda+it)}}{x_1-x_0}-\frac{e^{u_1(2\lambda+it)}-e^{u_0(2\lambda+it)}}{u_1-u_0}\right|\left(4\lambda^2+t^2\right)^{-1}\nonumber\\
&\qquad \times \left(\frac{e^{2x_1\lambda}+e^{2x_0\lambda}}{x_1-x_0}+\frac{e^{2u_1\lambda}+e^{2u_0\lambda}}{u_1-u_0}\right)^{-1}.
\end{align}
When $c\ge 7$ we use a simple estimate:
\begin{align}
G_{2,c}&\le 2\sup_{\substack{\lambda_{(1)},\lambda_{(2)}\ge \lambda_{min}\\ c-1\le t\le c+1}}\left((\lambda_{(1)}+\lambda_{(2)})^2+t^2\right)^{-1}\nonumber\\
&\le \frac{2}{4\lambda_{min}^2+(c-1)^2}.
\end{align}
For given constants $x_1,x_0,u_1,u_0,w,v$ and $\lambda_{min}$ we use \textit{Mathematica's} NMaximize function to calculate the bounds above for $G_{2,0}$ and $G_{2,c}$ for $1\le c\le 6$. We can estimate the bound given for $G_{2,c}$ when $7\le c\le 101$ exactly, and then for $c\ge 102$ we use an integral comparison to see that
\begin{align}
\sum_{c\ge 102}G_{2,c}&\le \sum_{m\ge 101}\frac{2}{4\lambda_{min}^2+m^2}\nonumber\\
&\le \int_{100}^\infty\frac{2}{4\lambda_{min}^2+t^2}dt\nonumber\\
&\le \frac{\tan^{-1}(\lambda_{min}/50)}{\lambda_{min}}.
\end{align}
We can then use this information to estimate $G_2$.
\begin{equation}G_2\le G_{2,0}+\sum_{1\le c\le 6}G_{2,c}+\sum_{6\le m\le 100}\frac{2}{4\lambda_{min}^2+m^2}+\frac{\tan^{-1}(\lambda_{min}/50)}{\lambda_{min}}.\end{equation}
As with the case in \cite{HB} it is optimal to choose $u_0=2w+1/3+\delta$ and $x_0=u_1+v+1/3+\delta$ with $\delta$ small. We will take $\delta=10^{-10}$ for our purposes. We are then left to choose suitable positive constants $w$, $u_1\ge u_0$, $v\ge u_1$ and $x_1\ge x_0$. We fix these now as
\begin{align}w=0.115 ,\quad u_0= 0.564, \quad u_1=0.620,\\
v=0.964,\quad x_0=1.413,\quad x_1=1.623.\end{align}
We consider $\lambda_{min}=0.35$. For this value we calculate that
\begin{equation}G_2\le 0.650.\end{equation}
Putting everything together we obtain
\begin{equation}\sum_{\substack{j,m\\ \lambda^{(j,m)}\ge 0.35}}\left(\frac{e^{3.246\lambda^{(j,m)}}+e^{2.826\lambda^{(j,m)}}}{0.210}+\frac{e^{1.240\lambda^{(j,m)}}+e^{1.128\lambda^{(j,m)}}}{0.056}\right)^{-1}\le 11.9288.\end{equation}
\subsection{Third Zero Density Estimate}
We now prove Lemma \ref{lmm:NewZeroDensity}. The proof uses the ideas \cite{HB}[Section 12] to obtain a stronger zero density estimate close to $1$, but agan we extend this to our slightly larger region with $\Im(\rho)\ll 1$. Specifically we wish to estimate
\begin{equation}N^*(\lambda):=\#\left\{\rho^{(j,m)}\in\mathcal{R}:\lambda^{(j,m)}\le \lambda\right\}\end{equation}
in the range $0\le \lambda\le 2$. We note that from the log-free zero density bound, that for $0\le \lambda\le 2$ we have that $N^*(\lambda)$ is uniformly bounded in $q$ and $\lambda$.

We adopt the notation of \cite{HB}. We put
\begin{equation}K(s,\chi):=\sum_{n=1}^\infty \Lambda(n)\Re\left(\frac{\chi(n)}{n^s}\right)g\left(\mathcal{L}^{-1}\log{n}\right)\end{equation}
for some function $g$ which satisfies:
\begin{cndtn}\label{Cndtn:NewCondition1}
$g:[0,\infty)\rightarrow \mathbb{R}$ is continuous, $g$ is supported on $[0,x_0)$ for some $x_0>0$, $g$ is twice differentiable on $(0,x_0)$ and $g''$ is bounded on $(0,x_0)$.
\end{cndtn}
\begin{cndtn}\label{Cndtn:NewCondition2}
$g$ is non-negative and its Laplace transform $G$ satisfies $\Re(G(z))\ge 0$ for $\Re(z)\ge 0$.
\end{cndtn}
We start with the following estimate
\begin{lmm}
Let $g$ be a function satisfying conditions 1 and 2 and let $\delta>0$. Then for $q>q_0(\delta,g)$ and $\lambda_1\ge \lambda_{11}$:

If
\begin{equation}\label{eq:NewZeroConditions}
G(\lambda-\lambda_{11})>g(0)/6\qquad \text{and} \qquad (G(\lambda-\lambda_{11})-g(0)/6)^2>G(-\lambda_{11})g(0)/6
\end{equation}
then we have
\begin{equation*}N^*(\lambda)\le \frac{G(-\lambda_{11})G_3}{(G(\lambda-\lambda_{11})-g(0)/6)^2-G(-\lambda_{11})g(0)/6}+\delta\end{equation*}
Where $G_3$ is defined in equation \eqref{eq:G3def}.
\end{lmm}
\begin{proof}
The first inequality of \cite{HB}[Section 12] shows that for $q>q_0(g,\delta_1)$ and $\beta_{11}=1-\lambda_{11}/\mathcal{L}$ we have
\begin{equation}\mathcal{L}^{-1}K(\beta_{11}+i\gamma^{(j,m)},\chi^{(j)})\le g(0)\phi_{\chi^{(j)}}/2+\delta_1-G(\lambda^{(j,m)}-\lambda_{11}).\end{equation}
Therefore, for any zero $\rho^{(j,m)}$ with $G(\lambda^{(j,m)}-\lambda_{11})>g(0)\phi_{\chi^{(j)}}/2$ we have
\begin{equation}0<G(\lambda^{(j,m)}-\lambda_{11})-g(0)\phi_{\chi^{(j)}}/2\le -\mathcal{L}^{-1}K(\beta_{11}+i\gamma^{(j,m)},\chi^{(j)})+\delta_1.\end{equation}
We note that $G(\lambda^{(j,m)}-\lambda_{11})$ is a decreasing function in $\lambda^{(j,m)}$ and recall that $\phi_\chi\le 1/3$ for all characters $\chi$. Therefore, if 
\begin{equation}\label{eq:NewZeroCondition}
G(\lambda-\lambda_{11})>g(0)/6,
\end{equation}
then for any $\lambda^{(j,m)}\le \lambda$ we have that
\begin{align}
0\le G(\lambda-\lambda_{11})-g(0)/6&\le G(\lambda^{(j,m)}-\lambda_{11})-g(0)\phi_{\chi^{(j)}}/2\nonumber\\
&\le -\mathcal{L}^{-1}K(\beta_{11}+i\gamma^{(j,m)},\chi^{(j)})+\delta_1.
\end{align}
We sum over all $j,m$ for which $\lambda^{(j,m)}\le \lambda$. Thus for $q>q_0(g,\delta_1)$ we have
\begin{align}
N^*(\lambda)(G(\lambda-&\lambda_{11})-g(0)/6)\nonumber\\
&\le \sum_{\substack{j,m\\ \lambda^{(j,m)}\le \lambda}}G(\lambda^{(j,m)}-\lambda_{11})-g(0)/6\nonumber\\
&\le -\mathcal{L}^{-1}\sum_{\substack{j,m\\ \lambda^{(j,m)}\le \lambda}}K(\beta_{11}+i\gamma^{(j,m)},\chi^{(j)})+\sum_{\substack{j,m\\ \lambda^{(j,m)}\le \lambda}}\delta_1\nonumber\\
&= -\mathcal{L}^{-1}\sum_{n=1}^\infty \Lambda(n)n^{-\beta_{11}}g(\mathcal{L}^{-1}\log{n})\Re\left(\sum_{\substack{j,m\\ \lambda^{(j,m)}\le \lambda}}\chi^{(j)}(n)n^{-i\gamma^{(j,m)}}\right)+\delta_2\nonumber\\
&\le \mathcal{L}^{-1}\sum_{n=1}^\infty\Lambda(n)n^{-\beta_{11}}\chi_0(n)g(\mathcal{L}^{-1}\log{n})\left|\sum_{\substack{j,m\\ \lambda^{(j,m)}\le \lambda}}\chi^{(j)}(n)n^{-i\gamma^{(j,m)}}\right|+\delta_2\nonumber\\
&\le\Sigma_1^{1/2}\Sigma_2^{1/2}+\delta_2\label{eq:MainNewIneq}
\end{align}
where
\begin{align}
\delta_2&=\sum_{\substack{j,m\\ \lambda^{(j,m)}\le \lambda}}\delta_1,\\
\Sigma_1&=\mathcal{L}^{-1}\sum_{n=1}^\infty\Lambda(n)n^{-\beta_{11}}\chi_0(n)g(\mathcal{L}^{-1}\log{n}),\\
\Sigma_2&=\mathcal{L}^{-1}\sum_{n=1}^\infty\Lambda(n)n^{-\beta_{11}}g(\mathcal{L}^{-1}\log{n})\left|\sum_{\substack{j,m\\ \lambda^{(j,m)}\le \lambda}}\chi^{(j)}(n)n^{-i\gamma^{(j,m)}}\right|^2.
\end{align}
By \cite{HB}[Lemma 5.3] for $q>q_0(g,\delta_1)$ we have
\begin{align}
\Sigma_1&=\mathcal{L}^{-1}K(\beta_{11},\chi_0)\nonumber\\ 
&\le G(-\lambda_{11})+\delta_1.\label{eq:NewSigma1}
\end{align}
We expand the square in $\Sigma_2$ and see that
\begin{equation}\Sigma_2=\Re(\Sigma_2)=\mathcal{L}^{-1}\sum_{\substack{j_1,j_2,m_1,m_2\\ \lambda^{(j_1,m_1)},\lambda^{(j_2,m_2)}\le \lambda}}K(\beta_{11}+i(\gamma^{(j_1,m_1)}-\gamma^{(j_2,m_2)}),\chi^{(j_1)}\overline{\chi}^{(j_2)}).\end{equation}
By \cite{HB}[Lemma 5.3] the terms with $j_1=j_2$ contribute a total
\begin{align}\mathcal{L}^{-1}\sum_{j_1,m_1,m_2}&K(\beta_{11}+i(\gamma^{(j_1,m_1)}-\gamma^{(j_1,m_2)}),\chi_0)\nonumber\\
&\le \sum_{j_1,m_1,m_2}\left(\left|\Re\left(G(-\lambda_{11}+i(v^{(j_1,m_1)}-v^{(j_1,m_2)}))\right)\right|+\delta_1\right).\end{align}
By \cite{HB}[Lemma 5.2] the terms with $j_1\ne j_2$ contribute
\begin{align}
\mathcal{L}^{-1}\sum_{j_1\ne j_2,m_1,m_2}K(\beta_{11}+i(\gamma^{(j_1,m_1)}-\gamma^{(j_2,m_2)}),\chi^{(j_1)}\overline{\chi}^{(j_2)})&\le \sum_{j_1\ne j_2,m_1,m_2}\left(g(0)/6+\delta_1\right).
\end{align}
Putting these together we get
\begin{align}\Sigma_2\le \sum_{\substack{j_1,m_1,m_2\\ \lambda^{(j_1,m_1)},\lambda^{(j_1,m_2)}\le \lambda}}\left(\left|\Re\left(G(-\lambda_{11}+i(\nu^{(j_1,m_1)}-\nu^{(j_1,m_2)}))\right)\right|-g(0)/6\right)\nonumber\\
+N^*(\lambda)^2g(0)/6+\delta_3.\end{align}
We put
\begin{equation}
G_3:=\sup_{j_1,m_1}\sum_{m_2}\left(\left|\Re\left(G(-\lambda_{11}+i(\gamma^{(j_1,m_1)}-\gamma^{(j_1,m_2)}))\right)\right|-g(0)/6\right),\label{eq:G3def}
\end{equation}
so
\begin{equation}\label{eq:NewSigma2}
\Sigma_2\le N^*(\lambda)^2g(0)/6+N^*(\lambda)G_3+\delta_3.
\end{equation}
Putting together \eqref{eq:MainNewIneq}, \eqref{eq:NewSigma1} and \eqref{eq:NewSigma2} we obtain
\begin{align}
N^*(\lambda)^2(G(\lambda-\lambda_{11})-g(0)/6)^2&\le \Sigma_1\Sigma_2+\delta_2\nonumber\\
&\le (G(-\lambda_{11})+\delta_1)(N^*(\lambda)^2g(0)/6+N^*(\lambda)G_3+\delta_3)+\delta_2.
\end{align}
Since $N^*(\lambda)$ is bounded uniformly for $0\le \lambda\le 2$ by the log-free zero density estimate, all the sums and terms are finite. Therefore, by a suitable choice of $\delta_1$ we have for given $\delta>0$ and $q>q_0(g,\delta)$ that
\begin{equation}N^*(\lambda)\left((G(\lambda-\lambda_{11})-g(0)/6)^2-G(-\lambda_{11})g(0)/6\right)^2\le G(-\lambda_{11})G_3+\delta\end{equation}
Therefore the lemma holds.
\end{proof}
We are now left to choose a suitable function $g$ and evaluate this expression. As in the work of Heath-Brown \cite{HB} and Xylouris \cite{Xylouris} we choose
\begin{equation}g(t):=\begin{cases}
\int^{\gamma}_{t-\gamma}(\gamma^2-x^2)(\gamma^2-(t-x)^2)dx\\
\qquad=-\frac{1}{30}t^5+\frac{2\gamma^2}{3}t^3-\frac{4\gamma^3}{3}t^2+\frac{16\gamma^5}{15},\qquad &t\in[0,2\gamma),\\
0,&t\ge 2\gamma,
\end{cases}\end{equation}
for some constant $\gamma>0$.

We see that $g$ is the convolution of $\max(0,\gamma^2-x^2)$ with itself, and so satisfies Condition 2 that $\Re(z)\ge 0\Rightarrow \Re(G(z))\ge 0$. We also see that it is twice differentiable on $(0,2\gamma)$ and its second derivative is continuous and bounded, and so also fulfills Condition 1.

We see the Laplace transform $G$ is
\begin{align}
G(z)&=\int_0^\infty e^{-z t}g(t)dt\nonumber\\
&=\begin{cases}
\frac{16\gamma^5}{15}z^{-1}-\frac{8\gamma^3}{3}z^{-3}+4\gamma^2(1+e^{-2\gamma z})z^{-4}\\
\qquad\qquad\qquad+4(-1+e^{-2\gamma z}+2\gamma z e^{-2\gamma z})z^{-6},\qquad &z\ne 0,\\
\frac{8\gamma^6}{9},&z=0.
\end{cases}
\end{align}
We bound $G_3$ in the same manner as we did in proving Lemma \ref{lmm:OldZeroDensity}. We recall
\begin{align}
G_3(\lambda)&=\sup_{m_1,j_1}\sum_{m_2}\left(\left|\Re(G(-\lambda_{11}+i(v^{(j_1,m_1)}-v^{(j_2,m_2)})))\right|-g(0)/6\right).
\end{align}
As in the proof of Lemma \ref{lmm:OldZeroDensity} we consider the contribution $G_{3,c}$ of zeros from rectangles $\mathcal{R}_{m_2}$ with $|m_1-m_2|=c\in\mathbb{Z}$.

We first consider $G_{3,0}$. There is only one zero $\rho^{(j_1,m_2)}=\rho^{(j_1,m_1)}$, if it exists. Thus
\begin{equation}G_{3,0}\le G(-\lambda_{11})-g(0)/6.\end{equation}
For $G_{3,c}$ with $1\le c \le 5$ we see that there are at most 2 zeros both with $c-1\le |v^{(j_1,m_1)}-v^{(j_1,m_2)}|\le c+1$. These contribute
\begin{equation}G_{3,c}\le 2\max\left(\sup_{c-1\le t\le c+1}|\Re(G(-\lambda_{11}+it))|-g(0)/6,0\right).\end{equation}
We estimate these using \textit{Mathematica's} NMaximize function.

We use a simpler bound to estimate $G_{3,c}$ with $c\ge 6$.
\begin{align}
|\Re(G(x+i y))|&\le \left|\frac{16\gamma^5}{15}\Re(z^{-1})\right|+\left|8\frac{\gamma^{3}}{3}\Re(z^{-3})\right|+4\gamma^{2}\left|\Re\left((1+e^{-2\gamma z})z^{-4}\right)\right|\nonumber\\
&\qquad\qquad+4\left|\Re\left((-1+e^{-2\gamma z}+2\gamma z e^{-2\gamma z})z^{-6}\right)\right|\nonumber\\
&\le \frac{16\gamma^5x}{15(x^2+y^2)}+\frac{8\gamma^3(|x|^3+3|x|y^2)}{3(x^2+y^2)^3}+\frac{4\gamma^2(1+e^{-2\gamma x})}{(x^2+y^2)^{2}}\nonumber\\
&\qquad\qquad+4(1+e^{-2\gamma x}+2\gamma (x^2+y^2)^{1/2}e^{-2\gamma x})(x^2+y^2)^{-3}\nonumber\\
&=:G_4(x,y).
\end{align}
We see that $G_4(x,y)$ is decreasing in $y$, and so
\begin{align}
G_{3,c}&\le 2\max\left(\sup_{c-1\le |t|\le c+1}\left|\Re(G(-\lambda_{11}+it))\right|-g(0)/6,0\right)\nonumber\\
&\le 2\max\left(G_4(-\lambda_{11},c-1)-g(0)/6,0\right).
\end{align}
We estimate this directly. We note that if $G_4(-\lambda_{11},c_1-1)\le g(0)/6$ then $G_{3,c}\le 0$ for all $c\ge c_1$.

Using these estimates we can then bound $G_3$ for any given value of our parameter $\gamma$ and a given lower bound for $\lambda_1$.

We consider separately the cases $\lambda_1\ge 0.35$, $\lambda_1\ge 0.40$, $\lambda_1\ge 0.44$, $\lambda_1\ge 0.52$, $\lambda_1\ge 0.60$, $\lambda_1\ge 0.66$ and $\lambda_1\ge 6/7$. In each case we choose $\gamma \in \{1.00,1.01,1.02, \dots,1.60\}$ which gives the best bound whilst ensuring that conditions \eqref{eq:NewZeroConditions} still hold.

We give the results in the following table. We note that in comparison with \cite{HB}[Table 13] these are worse by a factor of approximately 4, but are counting the number of rectangles containing a zero rather than just the number of characters.
\newpage
\begin{center}
\begin{longtable}{|c|c|c|c|c|c|c|c|}
\caption{Third Zero Density Estimate} \label{Table:NewZeroDensityTable}\\
\hline
$\lambda$ &  \multicolumn{7}{|c|}{Bound for $N^*(\lambda)$}\\
\cline{2-8}
 &$0.35\le\lambda_1$&$0.40\le\lambda_1$&$0.44\le\lambda_1$&	$0.52\le\lambda_1$&$0.60\le\lambda_1$&0.66 $\le\lambda_1$&$6/7\le\lambda_1$\\\hline
\endfirsthead
\hline
$\lambda$ &  \multicolumn{7}{|c|}{Bound for $N^*(\lambda)$ \textit{- continued from previous page}}\\
\cline{2-8}
 &$0.35\le\lambda_1$&$0.40\le\lambda_1$&$0.44\le\lambda_1$&	$0.52\le\lambda_1$&$0.60\le\lambda_1$&0.66 $\le\lambda_1$&$6/7\le\lambda_1$\\
\hline
\endhead
\hline
\multicolumn{8}{|c|}{\textit{Continued on next page...}}\\
\hline
\endfoot
\endlastfoot
0.74	 & 30	 & 29	 & 28	 & 27	 & 26	 & 26	 & - \\
0.75	 & 31	 & 30	 & 29	 & 28	 & 27	 & 26	 & - \\
0.76	 & 32	 & 31	 & 30	 & 29	 & 28	 & 27	 & - \\
0.77	 & 33	 & 32	 & 31	 & 30	 & 29	 & 28	 & - \\
0.78	 & 34	 & 33	 & 32	 & 31	 & 29	 & 29	 & - \\
0.79	 & 35	 & 34	 & 33	 & 32	 & 30	 & 29	 & - \\
0.80	 & 36	 & 35	 & 34	 & 32	 & 31	 & 30	 & - \\
0.81	 & 37	 & 36	 & 35	 & 33	 & 32	 & 31	 & - \\
0.82	 & 38	 & 37	 & 36	 & 34	 & 33	 & 32	 & - \\
0.83	 & 40	 & 38	 & 37	 & 35	 & 34	 & 33	 & - \\
0.84	 & 41	 & 39	 & 38	 & 37	 & 35	 & 34	 & - \\
0.85	 & 42	 & 41	 & 40	 & 38	 & 36	 & 35	 & - \\
0.86	 & 44	 & 42	 & 41	 & 39	 & 37	 & 36	 & - \\
0.87	 & 45	 & 44	 & 42	 & 40	 & 38	 & 37	 & 34 \\
0.88	 & 47	 & 45	 & 44	 & 41	 & 39	 & 38	 & 35 \\
0.89	 & 49	 & 47	 & 45	 & 43	 & 41	 & 39	 & 36 \\
0.90	 & 51	 & 49	 & 47	 & 44	 & 42	 & 40	 & 37 \\
0.91	 & 53	 & 50	 & 49	 & 46	 & 43	 & 42	 & 38 \\
0.92	 & 55	 & 52	 & 51	 & 47	 & 45	 & 43	 & 39 \\
0.93	 & 57	 & 54	 & 52	 & 49	 & 46	 & 44	 & 40 \\
0.94	 & 59	 & 57	 & 55	 & 51	 & 48	 & 46	 & 41 \\
0.95	 & 62	 & 59	 & 57	 & 53	 & 49	 & 47	 & 43 \\
0.96	 & 65	 & 61	 & 59	 & 55	 & 51	 & 49	 & 44 \\
0.97	 & 68	 & 64	 & 61	 & 57	 & 53	 & 51	 & 45 \\
0.98	 & 71	 & 67	 & 64	 & 59	 & 55	 & 52	 & 47 \\
0.99	 & 74	 & 70	 & 67	 & 61	 & 57	 & 54	 & 48 \\
1.00	 & 78	 & 73	 & 70	 & 64	 & 59	 & 56	 & 50 \\
1.01	 & 82	 & 77	 & 73	 & 67	 & 62	 & 58	 & 51 \\
1.02	 & 86	 & 80	 & 76	 & 70	 & 64	 & 61	 & 53 \\
1.03	 & 91	 & 84	 & 80	 & 73	 & 67	 & 63	 & 55 \\
1.04	 & 96	 & 89	 & 84	 & 76	 & 70	 & 66	 & 57 \\
1.05	 & 101	 & 94	 & 88	 & 80	 & 73	 & 68	 & 59 \\
1.06	 & 108	 & 99	 & 93	 & 83	 & 76	 & 71	 & 61 \\
1.07	 & 114	 & 105	 & 98	 & 88	 & 79	 & 74	 & 63 \\
1.08	 & 122	 & 111	 & 104	 & 92	 & 83	 & 78	 & 65 \\
1.09	 & 131	 & 118	 & 110	 & 97	 & 87	 & 81	 & 68 \\
1.10	 & 141	 & 127	 & 117	 & 103	 & 91	 & 85	 & 71 \\
1.11	 & 152	 & 136	 & 125	 & 109	 & 96	 & 89	 & 73 \\
1.12	 & 164	 & 146	 & 134	 & 115	 & 101	 & 94	 & 76 \\
1.13	 & 179	 & 157	 & 143	 & 122	 & 107	 & 98	 & 80 \\
1.14	 & 197	 & 171	 & 155	 & 130	 & 113	 & 104	 & 83 \\
1.15	 & 218	 & 186	 & 167	 & 139	 & 120	 & 110	 & 87 \\
1.16	 & 243	 & 205	 & 182	 & 150	 & 128	 & 116	 & 91 \\
1.17	 & 274	 & 226	 & 199	 & 161	 & 136	 & 123	 & 95 \\
1.18	 & 313	 & 253	 & 220	 & 175	 & 146	 & 131	 & 100 \\
1.19	 & 365	 & 286	 & 244	 & 190	 & 156	 & 140	 & 105 \\
1.20	 & 435	 & 328	 & 274	 & 208	 & 169	 & 149	 & 110 \\
1.21	 & 536	 & 383	 & 312	 & 229	 & 183	 & 160	 & 116 \\
1.22	 & 695	 & 458	 & 361	 & 255	 & 199	 & 173	 & 123 \\
1.23	 & 981	 & 568	 & 426	 & 286	 & 218	 & 187	 & 130 \\
1.24	 & 1642	 & 742	 & 518	 & 326	 & 241	 & 203	 & 138 \\
1.25	 & 4835	 & 1063	 & 658	 & 377	 & 268	 & 222	 & 146 \\
1.26	 & $\infty$	 & 1844	 & 895	 & 446	 & 301	 & 245	 & 156 \\
1.27	 & 	 & 6602	 & 1382	 & 543	 & 343	 & 272	 & 167 \\
1.28	 & 	 & $\infty$	 & 2967	 & 690	 & 397	 & 305	 & 179 \\
1.29	 & 	 & 	 & $\infty$	 & 940	 & 470	 & 347	 & 193 \\
1.30	 & 	 & 	 & 	 & 1457	 & 573	 & 400	 & 208 \\
1.31	 & 	 & 	 & 	 & 3156	 & 729	 & 471	 & 226 \\
1.32	 & 	 & 	 & 	 & $\infty$	 & 995	 & 569	 & 247 \\
1.33	 & 	 & 	 & 	 & 	 & 1549	 & 716	 & 272 \\
1.34	 & 	 & 	 & 	 & 	 & 3398	 & 958	 & 302 \\
1.35	 & 	 & 	 & 	 & 	 & $\infty$	 & 1433	 & 338 \\
1.36	 & 	 & 	 & 	 & 	 & 	 & 2782	 & 382 \\
1.37	 & 	 & 	 & 	 & 	 & 	 & 35205	 & 438 \\
1.38	 & 	 & 	 & 	 & 	 & 	 & $\infty$	 & 513 \\
1.39	 & 	 & 	 & 	 & 	 & 	 & 	 & 614 \\
1.40	 & 	 & 	 & 	 & 	 & 	 & 	 & 763 \\
1.41	 & 	 & 	 & 	 & 	 & 	 & 	 & 998 \\
1.42	 & 	 & 	 & 	 & 	 & 	 & 	 & 1430 \\
1.43	 & 	 & 	 & 	 & 	 & 	 & 	 & 2480 \\
1.44	 & 	 & 	 & 	 & 	 & 	 & 	 & 8791 \\
1.45	 & 	 & 	 & 	 & 	 & 	 & 	 & $\infty$ \\
\hline
\end{longtable}
\end{center}
\newpage
\section{Proof of Proposition \ref{Prpstn:OverallResult}}
We wish to estimate
\begin{equation*}\sum_{\chi\ne\chi_0}\sum_{m\in\mathbb{Z}}\sum_{\rho\in\mathcal{R}_m\cap\mathcal{Z}(\chi)}\exp(-M\lambda_\rho).\end{equation*}
We do this by Lemmas \ref{lmm:FirstZeroDensity}, \ref{lmm:OldZeroDensity} and \ref{lmm:NewZeroDensity}.

We split the argument into 2 sections, when there is a zero close to one (in which case it must be a real zero from a real character) and when there are no zeros close to one (and so $\rho_1$ or $\chi_1$ might be complex).

The work in this section follows along the same lines as that of \cite{HB}[Sections 14 and 15].
\subsection{A Zero close to 1}
We consider the case when $\eta\le \lambda_1\le 0.35$. By \cite{Xylouris}[Tabelle 11] we see that such a zero cannot exist if $\chi_1$ or $\rho_1$ is complex, and hence $\rho_1$ must be a real zero corresponding to a real character. Moreover, $\rho_1$ is simple. Since $\chi_1$ is real we have that $\phi_{\chi_1}=1/4$.

We first consider the contribution from characters $\chi^{(j)}\ne\chi_1$.

We note that
\begin{equation}\frac{\exp(-M\lambda)}{B_1(\lambda)}=\left(\frac{\lambda}{\sinh(K\lambda/2)}\right)^2\left(1+\frac{1}{4\lambda^2}\right)e^{-(M-K)\lambda}.\end{equation}
The first two terms in the product are decreasing in $\lambda$, and so for $M\ge K$ this is a decreasing function of $\lambda$. Therefore for all $\rho\in\mathcal{R}_m\cap\mathcal{Z}(\chi^{(j)})$, if $M\ge K$, we have
\begin{equation}\exp(-M\lambda_\rho)\le \frac{\exp(-M\lambda^{(j,m)})}{B_1(\lambda^{(j,m)})}B_1(\lambda_\rho).\end{equation}
Thus by Lemma \ref{lmm:FirstZeroDensity} we have
\begin{align}
\sum_{\rho\in\mathcal{R}_m\cap\mathcal{Z}(\chi^{(j)})}\exp(-M\lambda_\rho)&\le \frac{\exp(-M\lambda^{(j,m)})}{B_1(\lambda^{(j,m)})}\sum_{\rho\in\mathcal{R}_m\cap\mathcal{Z}(\chi^{(j)})}B_1(\lambda_\rho)\nonumber\\
&\le \frac{\exp(-M\lambda^{(j,m)})C_1(\lambda^{(j,m)})}{B_1(\lambda^{(j,m)})}.
\end{align}
We note that
\begin{equation*}\frac{\exp(-2x_1\lambda)}{B_2(\lambda)}\qquad \text{and}\qquad C_1(\lambda)\end{equation*}
are a decreasing functions in $\lambda$. Thus for $M\ge 2x_1+K$ we have that
\begin{equation}\frac{\exp(-M\lambda)C_1(\lambda)}{B_1(\lambda)B_2(\lambda)}\end{equation}
is a decreasing function in $\lambda$. Since for $\chi^{(j)}\ne \chi_1$ we have $\lambda^{(j,m)}\ge \lambda_2$, this gives us
\begin{equation}\label{chinot1}
\sum_{\substack{j,m\\ \chi^{(j)}\ne\chi_1,\chi_0}}\sum_{\rho\in\mathcal{R}_m\cap\mathcal{Z}(\chi^{(j)})}\exp(-M\lambda_\rho)\le \frac{\exp(-M\lambda_2)C_1(\lambda_2)}{B_2(\lambda_2)B_1(\lambda_2)}\sum_{\substack{j,m\\ \chi^{(j)}\ne\chi_1,\chi_0}}B_2(\lambda^{(j,m)}).
\end{equation}
We now consider the contribution from the character $\chi_1$. We give the zero $\rho_1$ close to 1 special treatment, and so treat the rectangle $\mathcal{R}_{0}$ which contains $\rho_1$ differently ($\rho_1\in\mathcal{R}_0$ since $\rho_1$ is real).

We first consider the contribution from rectangles $\mathcal{R}_m$ with $m\ne0$. Using the same ideas as above we have
\begin{equation}\label{chi1mnot0}
\sum_{m\ne 0}\sum_{\rho\in\mathcal{R}_m\cap\mathcal{Z}(\chi_1)}\exp(-M\lambda_\rho)\le \frac{\exp(-M\lambda_1')C_1(\lambda_1')}{B_2(\lambda_1')B_1(\lambda_1')}\sum_{\substack{m\ne0\\ \chi^{(j)}=\chi_1}}B_2(\lambda^{(j,m)}).
\end{equation}
We now consider the rectangle $\mathcal{R}_0$. We have
\begin{align}
\sum_{\rho\in\mathcal{R}_0\cap\mathcal{Z}(\chi_1)}\exp(-M\lambda_\rho)&\le \exp(-M\lambda_1) +\frac{\exp(-M\lambda_1')}{B_1(\lambda_1')}\sum_{\substack{\rho\in\mathcal{R}_0\cap\mathcal{Z}(\chi)\\ \rho\ne\rho_1}}B_1(\lambda_\rho)\nonumber\\
&\le \exp(-M\lambda_1)+\frac{\exp(-M\lambda_1')}{B_1(\lambda_1')}\sum_{\rho\in\mathcal{R}_0\cap\mathcal{Z}(\chi)}B_1(\lambda_\rho)\nonumber\\
&\le \exp(-M\lambda_1)+\frac{\exp(-M\lambda_1')C_1(\lambda_1)}{B_1(\lambda_1')}.
\end{align}
We note that $B_2(\lambda)$ and $C_1(\lambda)$ are both decreasing in $\lambda$. Therefore
\begin{equation}\sum_{\rho\in\mathcal{R}_0\cap\mathcal{Z}(\chi_1)}\exp(-M\lambda_\rho)\le \exp(-M\lambda_1)+\left(\frac{\exp(-M\lambda_1')C_1(0)}{B_1(\lambda_1')B_2(\lambda_1')}\right)B_2(\lambda_1).\end{equation}
Combining this with \eqref{chi1mnot0} and using the fact the $C_1$ is decreasing we obtain
\begin{equation}\label{chi1}
\sum_{\substack{j,m\\ \chi^{(j)}=\chi_1}}\sum_{\rho\in\mathcal{R}_m\cap\mathcal{Z}(\chi_1)}\exp(-M\lambda_\rho)\le\exp(-M\lambda_1)+ \frac{\exp(-M\lambda_1')C_1(0)}{B_1(\lambda_1')B_2(\lambda_1')}\sum_{\substack{j,m\\ \chi^{(j)}=\chi_1}}B_2(\lambda^{(j,m)}).
\end{equation}
Now combining \eqref{chi1} and \eqref{chinot1} we get
\begin{align}\sum_{\chi\ne\chi_0}\sum_{\rho\in\mathcal{R}\cap\mathcal{Z}(\chi)}\exp(-M\lambda_\rho)&\le \exp(-M\lambda_1)+C_4(\lambda_1',\lambda_2)\sum_{j,m}B_2(\lambda^{(j,m)})\nonumber\\
&\le \exp(-M\lambda_1)+C_4(\lambda_1',\lambda_2)C_2,
\end{align}
where
\begin{equation}C_4(\lambda_1',\lambda_2)=\max\left(\frac{\exp(-M\lambda_2)C_1(\lambda_2)}{B_1(\lambda_2)B_2(\lambda_2)},\frac{\exp(-M\lambda_1')C_1(0)}{B_1(\lambda_1')B_2(\lambda_1')}\right).\end{equation}
By \cite{HB}[Lemmas 8.4 and 8.8] for any $\delta>0$ and for all $q\ge q_0(\delta)$ we have
\begin{equation}\lambda_1',\lambda_2\ge \left(\frac{12}{11}-\delta\right)\log(\lambda_1^{-1}).\end{equation}
Also by \cite{HB}[Tables 4 and 7] for $\lambda_1\le 0.35$ we have that
\begin{equation}\lambda_1'\ge 2.19,\qquad\lambda_2\ge 1.42.\end{equation}
Thus, since $C_4(\lambda_1',\lambda_2)$ is decreasing in $\lambda_1'$ and $\lambda_2$, we have for any constant $B$ with $0\le B\le M-K-2x_1$
\begin{align}C_4(\lambda_1',\lambda_2)&\le \exp\left(-\left(\frac{12}{11}-\delta\right)B\log(\lambda_1^{-1})\right)\nonumber\\
&\qquad\times\max\left(\frac{\exp(-(M-B)\times 1.42)C_1(1.42)}{B_1(1.42)B_2(1.42)},\frac{\exp(-(M-B)\times2.19)C_1(0)}{B_1(2.19)B_2(2.19)}\right).
\end{align}
We choose
\begin{equation}B=1, \delta=0.01, K=0.66\end{equation}
and as before
\begin{align}w=0.115 ,\quad u_0= 0.564, \quad u_1=0.620,\\
v=0.964,\quad x_0=1.413,\quad x_1=1.623.\end{align}
Given $M$ we can now explicitly calculate the above quantities. For $M=7.5$ we obtain
\begin{equation}\sum_{\chi\ne\chi_0}\sum_{\rho\in\mathcal{R}\cap\mathcal{Z}(\chi)}\exp(7.5\lambda_\rho)\le \exp(-7.5\lambda_1)+2.38\times \lambda_1^{1.08}.\end{equation}
We see that the right hand side is a function which is 1 when $\lambda_1=0$, and is decreasing at $0$. Moreover, it is convex (has positive second derivative) on $(0,\infty)$ and so can have at most one turning point, which would be a minimum should it exist. Therefore the right hand side is always $<1$ for $\lambda_1\in [\eta,0.35]$ if it is $<1$ at 0.35.

Calculating this at 0.35 with $M=7.5$ gives 0.8628.., and so this is $<1$ for $\lambda_1\in[\eta,0.35]$ provided $M\ge 7.5$.

\subsection{No Zeroes close to 1}
We now consider the case when $\lambda_1\ge 0.35$.

As above, for characters $\chi^{(j)}\ne\chi_1,\overline{\chi}_1$ we have
\begin{align}
\sum_{\rho\in\mathcal{R}\cap\mathcal{Z}(\chi^{(j)})}\exp(-M\lambda_\rho)&\le  \sum_m\frac{\exp(-M\lambda^{(j,m)})}{B_1(\lambda^{(j,m)})}\sum_{\rho\in\mathcal{R}_m\cap\mathcal{Z}(\chi^{(j)})}B_1(\lambda_\rho)\nonumber\\
&\le \sum_m\frac{\exp(-M\lambda^{(j,m)})C_1(\lambda^{(j,m)})}{B_1(\lambda^{(j,m)})}.\label{eq:NotChi1}
\end{align}
We now consider the contributions for the character $\chi_1$ (and $\overline{\chi_1}$ if $\chi_1$ complex). We separate out the contribution of $\rho_1$ (and $\overline{\rho}_1$ if it exists). To do this we put
\begin{align}
n_1(\chi_1)&=
\begin{cases}
2,\quad &\chi_1 \text{ complex}\\
1,&\text{otherwise}
\end{cases}\\
n_2(\chi_1)&=
\begin{cases}
2,\quad &\chi_1 \text{ real and }\rho_1 \text{ complex}\\
1,&\text{otherwise}
\end{cases}\\
n_3(\chi_1)&=
\begin{cases}
2,\quad &\chi_1 \text{ real and }\rho_1 \text{ complex and }\rho_1\notin\mathcal{R}_0\\
1,&\text{otherwise.}
\end{cases}
\end{align}
We then have
\begin{equation}\sum_{\rho\in\mathcal{R}\cap\mathcal{Z}(\chi_1)}\exp(-M\lambda_1)=n_2(\chi_1)\exp(-M\lambda_\rho)+  \sum_m\sum_{\substack{\rho\in\mathcal{R}_m\cap\mathcal{Z}(\chi_1)\\ \rho\ne\rho_1,\overline{\rho}_1}}\exp(-M\lambda_\rho).\end{equation}
We separate out the contribution from the rectangle $\mathcal{R}_{m_1}$ which contains $\rho_1$. If $\chi_1$ is real and $\rho_1$ is complex then we also separate the rectangle $\mathcal{R}_{m_2}$ which contains $\overline{\rho}_1$ if this is different to $\mathcal{R}_{m_1}$. We note that all zeros in either of these rectangles have either $\lambda_\rho=\lambda_1$ or $\lambda_\rho\ge \lambda_1'$. The zeros in any other rectangle $\mathcal{R}_m$ have $\lambda_\rho\ge \lambda^{(j,m)}$. We then use Lemma \ref{lmm:FirstZeroDensity} again. This gives
\begin{align}
\sum_{\rho\in\mathcal{R}\cap\mathcal{Z}(\chi_1)}\exp(-M\lambda_\rho)&=n_2(\chi_1)\exp(-M\lambda_1)+\sum_{\substack{\rho\in(\mathcal{R}_{m_1}\cup\mathcal{R}_{m_2})\cap\mathcal{Z}(\chi_1)\\ \rho\ne\rho_1, \overline{\rho}_1}}\exp(-M\lambda_\rho)\nonumber\\
&\qquad\qquad\qquad+\sum_{m\ne m_1,m_2} \sum_{\rho\in\mathcal{R}_m\cap\mathcal{Z}(\chi_1)}\exp(-M\lambda_\rho)\nonumber\\
&\le \sum_{m\ne m_1,m_2} \frac{\exp(-M\lambda^{(j,m)}) C_1(\lambda^{(j,m)})}{B_1(\lambda^{(j,m)})}\nonumber\\
&\qquad\qquad +\frac{\exp(-M\lambda_1')}{B_1(\lambda_1')}\sum_{\rho\in(\mathcal{R}_{m_1}\cup\mathcal{R}_{m_2})\cap\mathcal{Z}(\chi_1)}B_1(\lambda_\rho)\nonumber\\
&\qquad\qquad+n_2(\chi_1)\left(\exp(-M\lambda_1)-\frac{\exp(-M\lambda_1')}{B_1(\lambda_1')}B_1(\lambda_1)\right)\nonumber\\
&\le \sum_{m\ne m_1,m_2} \frac{\exp(-M\lambda^{(j,m)})C_1(\lambda^{(j,m)})}{B_1(\lambda^{(j,m)})} +n_3(\chi_1)\frac{\exp(-M\lambda_1')C_1(\lambda_1)}{B_1(\lambda_1')}\nonumber\\
&\qquad\qquad+n_2(\chi_1)\left(\exp(-M\lambda_1)-\frac{\exp(-M\lambda_1')}{B_1(\lambda_1')}B_1(\lambda_1)\right)\nonumber\\
&\le \left(n_2(\chi_1)B_1(\lambda_1)-n_3(\chi_1)C_1(\lambda_1)\right)\left(\frac{\exp(-M\lambda_1)}{B_1(\lambda_1)}-\frac{\exp(-M\lambda_1')}{B_1(\lambda_1')}\right) \nonumber\\
&\qquad\qquad+\sum_{m}\frac{\exp(-M\lambda^{(j,m)})C_1(\lambda^{(j,m)})}{B_1(\lambda^{(j,m)})}.\label{eq:Chi1}
\end{align}
If $\chi_1$ is complex we follow the same argument and obtain the same result for $\overline{\chi}_1$.

Putting together \eqref{eq:NotChi1} and \eqref{eq:Chi1} we obtain
\begin{equation}\label{applied3}
\sum_{\chi\ne\chi_0}\sum_{\rho\in\mathcal{R}\cap\mathcal{Z}(\chi)}\exp(-M\lambda_\rho)\le\sum_{m,j}\frac{\exp(-M\lambda^{(j,m)})C_1(\lambda^{(j,m)})}{B_1(\lambda^{(j,m)})}+A_1,
\end{equation}
where
\begin{equation}A_1=n_1(\chi_1)\left(n_2(\chi_1)B_1(\lambda_1)-n_3(\chi_1)C_1(\lambda_1)\right)\left(\frac{\exp(-M\lambda_1)}{B_1(\lambda_1)}-\frac{\exp(-M\lambda_1')}{B_1(\lambda_1')}\right).\end{equation}
We now use Lemmas \ref{lmm:OldZeroDensity} and \ref{lmm:NewZeroDensity} to estimate the sum on the right hand side of \eqref{applied3}. We fix a constant $\Lambda$ (to be declared later) and consider separately the terms with $\lambda^{(j,m)}>\Lambda$ and $\lambda^{(j,m)}\le \Lambda$. We use Lemma \ref{lmm:OldZeroDensity} to estimate the first set of terms, and Lemma \ref{lmm:NewZeroDensity} to estimate the second set.

We first consider the terms with $\lambda^{(j,m)}>\Lambda$.
\begin{equation}\sum_{\substack{j,m\\ \lambda^{(j,m)}>\Lambda}}\frac{\exp(-M\lambda^{(j,m)})C_1(\lambda^{(j,m)})}{B_1(\lambda^{(j,m)})}=\sum_{j,m}\left(\frac{\exp(-M\lambda^{(j,m)})C_1(\lambda^{(j,m)})}{B_1(\lambda^{(j,m)})B_2(\lambda^{(j,m)})}\right)B_2(\lambda^{(j,m)}).\end{equation}
Again we note that
\begin{equation*}\frac{\exp(-K\lambda)}{B_1(\lambda)},\qquad \frac{\exp(-2x_1\lambda)}{B_2(\lambda)},\qquad\text{and}\qquad C_1(\lambda)\end{equation*}
are all decreasing functions of $\lambda$. Therefore, provided $M\ge K+2x_1$ we have
\begin{align}
\sum_{\substack{j,m\\ \lambda^{(j,m)}>\Lambda}}&\frac{\exp(-M\lambda^{(j,m)})C_1(\lambda^{(j,m)})}{B_1(\lambda^{(j,m)})}\nonumber\\
&\qquad\le \left(\frac{\exp(-M\Lambda)C_1(\Lambda)}{B_1(\Lambda)B_2(\Lambda)}\right)\sum_{\substack{j,m\\  \lambda^{(j,m)}>\Lambda}}B_2(\lambda^{(j,m)})\nonumber\\
&\qquad=\left(\frac{\exp(-M\Lambda)C_1(\Lambda)}{B_1(\Lambda)B_2(\Lambda)}\right)\sum_{j,m}B_2(\lambda^{(j,m)})\nonumber\\
&\qquad\qquad-\left(\frac{\exp(-M\Lambda)C_1(\Lambda)}{B_1(\Lambda)B_2(\Lambda)}\right)\sum_{\substack{j,m\\ \lambda^{(j,m)}\le\Lambda}}B_2(\lambda^{(j,m)})\nonumber\\
&\qquad\le \frac{\exp(-M\Lambda)C_1(\Lambda)C_2}{B_1(\Lambda)B_2(\Lambda)}-\left(\frac{\exp(-M\Lambda)C_1(\Lambda)}{B_1(\Lambda)B_2(\Lambda)}\right)\sum_{\substack{j,m\\ \lambda^{(j,m)}\le\Lambda}}B_2(\lambda^{(j,m)}).
\end{align}
Hence
\begin{align}
\sum_{m,j}&\frac{\exp(-M\lambda^{(j,m)})C_1(\lambda^{(j,m)})}{B_1(\lambda^{(j,m)})} \le\frac{\exp(-M\Lambda)C_1(\Lambda)C_2}{B_1(\Lambda)B_2(\Lambda)}\nonumber\\ 
&\qquad\qquad+\sum_{\substack{j,m\\ \lambda^{(j,m)}\le\Lambda}}\left(\frac{\exp(-M\lambda^{(j,m)})C_1(\lambda^{(j,m)})}{B_1(\lambda^{(j,m)})B_2(\lambda^{(j,m)})}-\frac{\exp(-M\Lambda)C_1(\Lambda)}{B_1(\Lambda)B_2(\Lambda)}\right)B_2(\lambda^{(j,m)})\label{applied1}.
\end{align}
We therefore are left to evaluate
\begin{equation}\sum_{\substack{j,m\\ \lambda^{(j,m)}\le\Lambda}}\left(\frac{\exp(-M\lambda^{(j,m)})C_1(\lambda^{(j,m)})}{B_1(\lambda^{(j,m)})B_2(\lambda^{(j,m)})}-\frac{\exp(-M\Lambda)C_1(\Lambda)}{B_1(\Lambda)B_2(\Lambda)}\right)B_2(\lambda^{(j,m)}).\end{equation}
To ease notation we put
\begin{equation}D(\lambda)=\left(\frac{\exp(-M\lambda)C_1(\lambda)}{B_1(\lambda)B_2(\lambda)}-\frac{\exp(-M\Lambda)C_1(\Lambda)}{B_1(\Lambda)B_2(\Lambda)}\right)B_2(\lambda).\end{equation}
We note that $D(\lambda)$ is a decreasing function of $\lambda$ (and is non-negative for $\lambda\le \Lambda$).

We separate the terms for $\lambda_1$ and put $\lambda^*=\min(\lambda_1',\lambda_2)$. This gives
\begin{equation}\sum_{\substack{j,m\\ \lambda^{(j,m)}\le \Lambda}}D(\lambda^{(j,m)})=n_3(\chi_1)n_1(\chi_1)D(\lambda_1)+\sum_{\substack{j,m\\  \lambda^*\le \lambda^{(j,m)}\le \Lambda}}D(\lambda^{(j,m)}).\end{equation}
We put $\Lambda_r=\Lambda-(0.01)r$ and define $s$ such that $\Lambda_{s+1}\le \lambda^*< \Lambda_s$. We then split the sum into a sum over the different ranges $\Lambda_{r+1}\le \lambda^{(j,m)}< \Lambda_r$.
\begin{align}
\sum_{\substack{j,m\\ \lambda^*\le \lambda^{(j,m)}\le \Lambda}}D(\lambda^{(j,m)})&\le \sum_{r=0}^{s-1}\sum_{\substack{j,m\\ \Lambda_{r+1}\le \lambda^{(j,m)}\le \Lambda_r}}D(\lambda^{(j,m)})+\sum_{\substack{j,m\\ \lambda^*\le \lambda^{(j,m)}\le \Lambda_s}}D(\lambda^{(j,m)})\nonumber\\
&\le (N^*(\Lambda_s)-n_1(\chi_1)n_3(\chi_1))D(\lambda^*)\nonumber\\
&\qquad+\sum_{r=0}^{s-1}(N^*(\Lambda_r)-N^*(\Lambda_{r+1}))D(\Lambda_{r+1}).
\end{align}
Note that we have used the fact that $D(\lambda)$ is decreasing in $\lambda$.

By Abel's identity we have
\begin{align}
\sum_{\substack{j,m\\ \lambda^*\le \lambda^{(j,m)}\le \Lambda}}D(\lambda^{(j,m)})&\le -n_1(\chi_1)n_3(\chi_1)D(\lambda^*)+N^*(\Lambda_s)(D(\lambda^*)-D(\Lambda_s))\nonumber\\
&\qquad\qquad\qquad+\sum_{r=0}^{s-1}N^*(\Lambda_r)(D(\Lambda_{r+1})-D(\Lambda_r)),
\end{align}
since $D(\Lambda)=0$.

Since $D(\Lambda_{r+1})\ge D(\Lambda_r)$ and $D(\lambda^*)\ge D(\Lambda_s)$ we may replace $N^*(\lambda)$ with an upper bound, say $N_0^*(\lambda)$. This gives
\begin{align}
\sum_{\substack{j,m\\ \lambda*\le \lambda^{(j,m)}\le \Lambda}}D(\lambda^{(j,m)})&\le -n_1(\chi_1)n_3(\chi_1)D(\lambda^*)+N^*_0(\Lambda_s)D(\lambda^*)\nonumber\\
&\qquad\qquad\qquad+\sum_{r=0}^{s-1}(N^*_0(\Lambda_r)-N^*_0(\Lambda_{r+1}))D(\Lambda_{r+1}).
\end{align}

Hence
\begin{align}
\sum_{\substack{j,m\\ \lambda^{(j,m)}\le \Lambda}}D(\lambda^{(j,m)})&\le n_1(\chi_1)n_3(\chi_1)(D(\lambda_1)-D(\lambda^*))+N^*_0(\Lambda_s)D(\lambda^*)\nonumber\\
&\qquad\qquad+\sum_{r=0}^{s-1}(N^*_0(\Lambda_r)-N^*_0(\Lambda_{r+1}))D(\Lambda_{r+1}).\label{applied2}
\end{align}
Putting \eqref{applied3}, \eqref{applied1} and \eqref{applied2} together we obtain
\begin{align}
\sum_{\chi\ne\chi_0}\sum_{\rho\in\mathcal{R}\cap\mathcal{Z}(\chi)}\exp(-M\lambda_\rho)&\le \frac{\exp(-M\Lambda)C_1(\Lambda)C_2}{B_1(\Lambda)B_2(\Lambda)}+\sum_{r=0}^{s-1}(N^*_0(\Lambda_r)-N^*_0(\Lambda_{r+1}))D(\Lambda_{r+1})\nonumber\\
&\qquad+N^*_0(\Lambda_s)D(\lambda^*)+A_1',\label{RHS}
\end{align}
where
\begin{align}
A_1'&=n_1(\chi_1)n_2(\chi_1)B_1(\lambda_1)\left(\frac{\exp(-M\lambda_1)}{B_1(\lambda_1)}-\frac{\exp(-M\lambda_1')}{B_1(\lambda_1')}\right)\nonumber\\
&+n_1(\chi_1)n_3(\chi_1)\left(D(\lambda_1)-D(\lambda^*)-C_1(\lambda_1)\left(\frac{\exp(-M\lambda_1)}{B_1(\lambda_1)}-\frac{\exp(-M\lambda_1')}{B_1(\lambda_1')}\right)\right).
\end{align}
We now wish to bound this when we consider $\lambda_1,\lambda_1'$ and $\lambda_2$ constrained in size. Specifically, we consider $\lambda_1\in[\lambda_{11},\lambda_{12}]$, $\lambda_2\ge\lambda_{21}$ and $\lambda_1'\ge \lambda_{11}'$.

By definition $N_0^*(\Lambda_s)\ge n_1(\chi_1)n_3(\chi_1)$, and so the coefficient of $D(\lambda^*)$ is $>0$. Since $D$ is a decreasing function, the right hand side of \eqref{RHS} is decreasing as a function of $\lambda_2$. The term $B_1(\lambda_1)$ occurs $n_2(\chi_1)/n_3(\chi_1)$ times in the sum
\begin{equation*}\sum_{\rho\in\mathcal{R}_0\cap\mathcal{Z}(\chi_1)}B_1(\lambda_\rho).\end{equation*}
Since the sum is $\le C_1(\lambda_1)$, and all terms in the sum are positive we have that
\begin{equation}n_2(\chi_1)B_1(\lambda_1)\le n_3(\chi_1)C_1(\chi_1).\end{equation}
Therefore, by expanding out $A'$ we see that the right hand side of \eqref{RHS} is also decreasing as a function of $\lambda_1'$.

Therefore we may replace them $\lambda_1'$ and $\lambda_2$ with their lower bounds $\lambda_{11}'$ and $\lambda_{21}$ respectively.

Considering this bound as a function of $\lambda_1$ we find that the right hand side is
\begin{align}
&n_1(\chi_1)n_2(\chi_1)B_1(\lambda_1)\left(\frac{\exp(-M\lambda_1)}{B_1(\lambda_1)}-\frac{\exp(-M\lambda_{11}')}{B_1(\lambda_{11}')}\right)\nonumber\\
&\qquad\qquad+n_1(\chi_1)n_3(\chi_1)\left(\frac{\exp(-M\lambda_{11}')}{B_1(\lambda_{11}')}C_1(\lambda_1)-\frac{\exp(-M\Lambda)C_1(\Lambda)}{B_1(\Lambda)B_2(\Lambda)}B_2(\lambda_1)\right)+C,
\end{align}
where $C$ is independent of $\lambda_1$. We see this is
\begin{align}
&\le 2B_1(\lambda_{11})\left(\frac{\exp(-M\lambda_{11})}{B_1(\lambda_{11})}-\frac{\exp(-M\lambda_{11}')}{B_1(\lambda_{11}')}\right)\nonumber\\
&\qquad\qquad+n_1(\chi_1)n_3(\chi_1)\left(\frac{\exp(-M\lambda_{11}')}{B_1(\lambda_{11}')}C_1(\lambda_{11})-\frac{\exp(-M\Lambda)C_1(\Lambda)}{B_1(\Lambda)B_2(\Lambda)}B_2(\lambda_{12})\right)+C.
\end{align}
Therefore we obtain
\begin{align}
\sum_{\chi\ne\chi_0}\sum_{\rho\in\mathcal{R}\cap\mathcal{Z}(\chi)}\exp(-M\lambda_\rho)&\le \frac{\exp(-M\Lambda)C_1(\Lambda)C_2}{B_1(\Lambda)B_2(\Lambda)}+\sum_{r=0}^{s-1}(N^*_0(\Lambda_r)-N^*_0(\Lambda_{r+1}))D(\Lambda_{r+1})\nonumber\\
&\qquad\qquad+N^*_0(\Lambda_s)D(\lambda^*)+A_1'',\label{FinalSetup}
\end{align}
where
\begin{align}
A_1''&=
2B_1(\lambda_{11})\left(\frac{\exp(-M\lambda_{11})}{B_1(\lambda_{11})}-\frac{\exp(-M\lambda_{11}')}{B_1(\lambda_{11}')}\right)\nonumber\\
&\qquad\qquad+n_4\left(\frac{\exp(-M\lambda_{11}')}{B_1(\lambda_{11}')}C_1(\lambda_{11})-\frac{\exp(-M\Lambda)C_1(\Lambda)}{B_1(\Lambda)B_2(\Lambda)}B_2(\lambda_{12})-D(\lambda^*)\right),
\end{align}
and $n_4$ is chosen to be 1 or 2 so as to give the largest value for $A_1''$.

We now proceed to estimate \eqref{FinalSetup} for various ranges of $\lambda_1$ which cover the region $\lambda_1\ge0.35$. We consider
\begin{equation}M=7.999.\end{equation}
For each range of $\lambda_1$ we use the lower bounds for $\lambda_1'$ and $\lambda_2$ as given by \cite{Xylouris}[Tabelle 2, 3, 7] and \cite{HB}[Table 4 and 7]. We use the upper bounds for $N_0^*$ as calculated in Table \ref{Table:NewZeroDensityTable}.

We give these bounds on $\lambda_1'$ and $\lambda_2$, our choices of $\Lambda$ and the calculation of the right hand side of \eqref{FinalSetup} in Table \ref{Table:ResultsTable}.

We see that for each range of $\lambda_1$ we obtain an upper bound for \eqref{FinalSetup} which is $<0.99$. Since the expression is decreasing in $M$, this holds for all $M\ge 7.8$. We have therefore established Proposition \ref{Prpstn:OverallResult} by taking $\epsilon=10^{-3}$.

\begin{table}[t]\label{Table:ResultsTable}
\begin{center}
\caption{Calculation of the RHS of \eqref{FinalSetup} for different ranges of $\lambda_1$.}
\begin{tabular}{|c|c|c|c|c|c|}
\hline
$\lambda_{11}$ & $\lambda_{12}$ & $\lambda_{21}$ & $\lambda_{11}'$ & $\Lambda$ & Total RHS of \eqref{FinalSetup}\\
\hline
0.35 & 0.40 & 1.29 & 2.10 & 1.29 & 0.8579...\\
0.40 & 0.44 & 1.18 & 2.03 & 1.27 & 0.9821...\\
0.44 & 0.46 & 1.08 & 1.66 & 1.28 & 0.9213...\\
0.46 & 0.48 & 1.08 & 1.53 & 1.28 & 0.9120...\\
0.48 & 0.50 & 1.08 & 1.47 & 1.28 & 0.9041...\\
0.50 & 0.52 & 1.00 & 1.40 & 1.28 & 0.9304...\\
0.52 & 0.54 & 1.00 & 1.34 & 1.31 & 0.8049...\\
0.54 & 0.56 & 0.92 & 1.28 & 1.31 & 0.8427...\\
0.56 & 0.58 & 0.92 & 1.23 & 1.31 & 0.8385...\\
0.58 & 0.60 & 0.92 & 1.18 & 1.31 & 0.8349...\\
0.60 & 0.62 & 0.85 & 1.13 & 1.34 & 0.7782...\\
0.62 & 0.64 & 0.85 & 1.09 & 1.34 & 0.7756...\\
0.64 & 0.66 & 0.79 & 1.04 & 1.34 & 0.8363...\\
0.66 & 0.68 & 0.79 & 1.00 & 1.36 & 0.7652...\\
0.68 & 0.70 & 0.79 & 0.96 & 1.36 & 0.7636...\\
0.70 & 0.72 & 0.745 & 0.93 & 1.36 & 0.8241...\\
0.72 & 0.74 & 0.745 & 0.91 & 1.36 & 0.8229...\\
0.74 & 0.76 & 0.745 & 0.89 & 1.36 & 0.8219...\\
0.76 & 0.78 & 0.76 & 0.86 & 1.36 & 0.7988...\\
0.78 & 0.80 & 0.78 & 0.84 & 1.36 & 0.7708..\\
0.80 & 0.82 & 0.80 & 0.83 & 1.36 & 0.7463...\\
0.82 & 0.86 & 0.82 & 0.827 & 1.36 & 0.7243...\\
0.86 & $\infty$ & 0.86 & 0.86 & 1.44 & 0.5110...\\
\hline
\end{tabular}
\end{center}
\end{table}
\section{Acknowledgment}
I would like to thank my supervisor, Prof. Heath-Brown, for suggesting this problem, for providing a huge number of helpful comments and for his encouragement.
\bibliographystyle{acm}
\bibliography{bibliography}

\begin{thebibliography}{10}

\bibitem{Chen1}
{\sc Chen, J.-R.}
\newblock On the least prime in an arithmetical progression.
\newblock {\em Sci. Sinica 14\/} (1965), 1868--1871.

\bibitem{Chen2}
{\sc Chen, J.~R.}
\newblock On the least prime in an arithmetical progression and two theorems
  concerning the zeros of {D}irichlet's {$L$}-functions.
\newblock {\em Sci. Sinica 20}, 5 (1977), 529--562.

\bibitem{Chen3}
{\sc Chen, J.~R.}
\newblock On the least prime in an arithmetical progression and theorems
  concerning the zeros of {D}irichlet's {$L$}-functions. {II}.
\newblock {\em Sci. Sinica 22}, 8 (1979), 859--889.

\bibitem{ChenLiu}
{\sc Chen, J.~R., and Liu, J.}
\newblock On the least prime in an arithmetical progression and theorems
  concerning the zeros of {D}irichlet's {$L$}-functions. {V}.
\newblock In {\em International {S}ymposium in {M}emory of {H}ua {L}oo {K}eng,
  {V}ol.\ {I} ({B}eijing, 1988)}. Springer, Berlin, 1991, pp.~19--42.

\bibitem{Iwaniec2}
{\sc Friedlander, J., and Iwaniec, H.}
\newblock The {B}run-{T}itchmarsh theorem.
\newblock In {\em Analytic number theory ({K}yoto, 1996)}, vol.~247 of {\em
  London Math. Soc. Lecture Note Ser.} Cambridge Univ. Press, Cambridge, 1997,
  pp.~85--93.

\bibitem{Goldfeld}
{\sc Goldfeld, D.~M.}
\newblock A further improvement of the {B}run-{T}itchmarsh theorem.
\newblock {\em J. London Math. Soc. (2) 11}, 4 (1975), 434--444.

\bibitem{Graham2}
{\sc Graham, S.}
\newblock An asymptotic estimate related to {S}elberg's sieve.
\newblock {\em J. Number Theory 10}, 1 (1978), 83--94.

\bibitem{Graham1}
{\sc Graham, S.}
\newblock On {L}innik's constant.
\newblock {\em Acta Arith. 39}, 2 (1981), 163--179.

\bibitem{GrahamThesis}
{\sc Graham, S.~W.}
\newblock {\em Applications of Sieve Methods}.
\newblock ProQuest LLC, Ann Arbor, MI, 1977.
\newblock Thesis (Ph.D.)--University of Michigan.

\bibitem{HB}
{\sc Heath-Brown, D.~R.}
\newblock Zero-free regions for {D}irichlet {$L$}-functions, and the least
  prime in an arithmetic progression.
\newblock {\em Proc. London Math. Soc. (3) 64}, 2 (1992), 265--338.

\bibitem{Iwaniec1}
{\sc Iwaniec, H.}
\newblock On the {B}run-{T}itchmarsh theorem.
\newblock {\em J. Math. Soc. Japan 34}, 1 (1982), 95--123.

\bibitem{Jutila1}
{\sc Jutila, M.}
\newblock A new estimate for {L}innik's constant.
\newblock {\em Ann. Acad. Sci. Fenn. Ser. A I No. 471\/} (1970), 8.

\bibitem{Jutila}
{\sc Jutila, M.}
\newblock On {L}innik's constant.
\newblock {\em Math. Scand. 41}, 1 (1977), 45--62.

\bibitem{Linnik1}
{\sc Linnik, U.~V.}
\newblock On the least prime in an arithmetic progression. {I}. {T}he basic
  theorem.
\newblock {\em Rec. Math. [Mat. Sbornik] N.S. 15(57)\/} (1944), 139--178.

\bibitem{Linnik2}
{\sc Linnik, U.~V.}
\newblock On the least prime in an arithmetic progression. {II}. {T}he
  {D}euring-{H}eilbronn phenomenon.
\newblock {\em Rec. Math. [Mat. Sbornik] N.S. 15(57)\/} (1944), 347--368.

\bibitem{MontgomeryBook}
{\sc Montgomery, H.~L.}
\newblock Problems concerning prime numbers.
\newblock In {\em Mathematical developments arising from {H}ilbert problems
  ({P}roc. {S}ympos. {P}ure {M}ath., {N}orthern {I}llinois {U}niv., {D}e
  {K}alb, {I}ll., 1974)}. Amer. Math. Soc., Providence, R. I., 1976,
  pp.~307--310. Proc. Sympos. Pure Math., Vol. XXVIII.

\bibitem{MontgomeryVaughan}
{\sc Montgomery, H.~L., and Vaughan, R.~C.}
\newblock The large sieve.
\newblock {\em Mathematika 20\/} (1973), 119--134.

\bibitem{MotohashiI}
{\sc Motohashi, Y.}
\newblock On some improvements of the {B}run-{T}itchmarsh theorem.
\newblock {\em J. Math. Soc. Japan 26\/} (1974), 306--323.

\bibitem{Pan}
{\sc Pan, C.~D.}
\newblock On the least prime in an arithmetical progression.
\newblock {\em Sci. Record (N.S.) 1\/} (1957), 311--313.

\bibitem{Pintz}
{\sc Pintz, J.}
\newblock Elementary methods in the theory of {$L$}-functions. {II}. {O}n the
  greatest real zero of a real {$L$}-function.
\newblock {\em Acta Arith. 31}, 3 (1976), 273--289.

\bibitem{Siegel}
{\sc Siegel, C.~L.}
\newblock \"{U}ber die {C}lassenzahl quadratischer {Z}ahlk\"orper.
\newblock {\em Acta Arith}, 1 (1936), 83--86.

\bibitem{Titchmarsh}
{\sc Titchmarsh, E.~C.}
\newblock A divisor problem.
\newblock {\em Rend. Circ. Math. Palermo 54\/} (1930), 414--429.

\bibitem{Walfisz}
{\sc Walfisz, A.}
\newblock Zur additiven {Z}ahlentheorie. {II}.
\newblock {\em Math. Z. 40}, 1 (1936), 592--607.

\bibitem{Wang}
{\sc Wang, W.}
\newblock On the least prime in an arithmetic progression.
\newblock {\em Acta Math. Sinica (N.S.) 7}, 3 (1991), 279--289.
\newblock A Chinese summary appears in Acta Math. Sinica {{\bf{3}}5} (1992),
  no. 4, 575.

\bibitem{Xylouris}
{\sc Xylouris, T.}
\newblock On the least prime in an arithmetic progression and estimates for the
  zeros of {D}irichlet {$L$}-functions.
\newblock {\em Acta Arith. 150}, 1 (2011), 65--91.

\end{thebibliography}
\end{document}